\documentclass[a4paper]{amsart}

\usepackage{amsmath,amssymb,amsthm,times,mathrsfs}
\usepackage{bbm}
\usepackage[all]{xy}
\usepackage[dvipdfmx,colorlinks=true,backref=page]{hyperref}

\numberwithin{equation}{section}

\theoremstyle{theorem}
\newtheorem{theorem}{Theorem}[section]
\newtheorem{proposition}[theorem]{Proposition}
\newtheorem{lemma}[theorem]{Lemma}
\newtheorem{corollary}[theorem]{Corollary}

\newtheorem{problem}[theorem]{Problem}
\newtheorem{question}[theorem]{Question}

\theoremstyle{definition}
\newtheorem{definition}[theorem]{Definition}
\newtheorem{example}[theorem]{Example}

\theoremstyle{remark}
\newtheorem{remark}[theorem]{Remark}

\newcommand{\Z}{\mathbb{Z}}

\newcommand{\R}{\mathbb{R}}

\newcommand{\cat}{\mathsf{cat}}
\newcommand{\secat}{\mathsf{secat}}
\newcommand{\TC}{\mathsf{TC}}
\newcommand{\wcat}{\mathsf{wcat}}

\SelectTips{cm}{}

\title{Parametrized topological complexity of spherical fibrations over spheres}

\author[Y. Minowa]{Yuki Minowa}
\address{Department of Mathematics, Kyoto University, Kyoto, 606-8502, Japan}
\email{minowa.yuki.48z@st.kyoto-u.ac.jp}

\date{today}



\begin{document}

\begin{abstract}
  Parametrized topological complexity is a homotopy invariant that represents the degree of instability of motion planning problem that involves external constraints. We consider the parametrized topological complexity in the case of spherical fibrations over spheres. We explicitly compute a lower bound in terms of weak category and determine the parametrized topological complexity of some spherical fibrations.
\end{abstract}

\maketitle

%%%%% Section 1 %%%%%

\section{Introduction}\label{Introduction}

Farber \cite{F} defined topological complexity as a quantity for topological study of the robot motion planning problem. The topological complexity of a space $X$ represents the degree of instability of motion planning in $X$ as a configuration space. Recently, Cohen, Farber and Weinberger \cite{CFW} introduced a variant of topological complexity, called \emph{parametrized topological complexity}, to formulate the robot motion planning problem that involves external constraints as parameters, such as a set of obstacles that block the motion of robots. The configuration space of this motion planning problem is described as a fibration $p\colon X\to B$, where $X$ represents the configuration space and $B$ represents the external constraints. Then parametrized topological complexity is defined by a topological invariant of the map $p\colon X\to B$, where we denote it by $\TC[X\to B]$. See Section \ref{Parametrized topological complexity} for the precise definition.

Initially, the studies of parametrized topological complexity was motivated by the computation of $\TC[F(\R^k,m+n)\to F(\R^k,m)]$, where $F(Y,m)$ denotes the configuration space
\[
\{(x_1,\cdots,x_m)\in Y^{m}|x_i\neq x_j\quad\text{if}\quad i\neq j\}.
\]
Several authors have attempted to determine or estimate the parametrized topological complexity of other specific fibrations. In particular, Farber and Weinberger \cite{FW0, FW1} studied the parametrized topological complexity of sphere bundles coming from vector bundles; they gave an upper bound in terms of the $2$-frame bundle and a lower bound in terms of characteristic classes. Crabb \cite{C} reformulate this lower bound in terms of stable cohomotopy Euler class and obtained a similar result for projective bundles coming from vector bundles. Farber and Oprea \cite{FO} considered a fiber bundle $F\to X\to B$ associated to a principal $G$-bundle; they gave an upper and lower bounds in terms of $G$-equivariant homotopy invariants.

On the other hand, many variants of parametrized topological complexity have also been investigated. Farber and Paul \cite{FP22} introduced sequential parametrized topological complexity as an analog of sequential topological complexity \cite{R}. They also gave upper and lower bounds in terms of $2$-frame bundle and characteristic classes \cite{FP23}. Daundkar \cite{D} considered equivariant parametrized topological complexity for a fibration $X\to B$ that is also a $G$-map. Grant \cite{Gr} defined the parametrized topological complexity of an epimorphism of discrete groups, which coincides with the parametrized topological complexity of the induced fibration between corresponding Eilenberg-MacLane spaces, he gave an upper and lower bound in terms of the cohomology dimension of groups. However, in this paper, we will not delve into these variants.

In this paper, we investigate the parametrized topological complexity of spherical fibrations over spheres. Let
\[
  S^n\to X\to S^{m+1}
\]
be a homotopy fibration with $n\ge 1$. In general, as in \cite{CFW}, there is a simple lower bound
\[
  \TC[X\to S^{m+1}]\ge\TC(S^n)=
  \begin{cases}
    1&n\text{ is odd}\\
    2&n\text{ is even.}
  \end{cases}
\]
This lower bound coincides with the cup-length type cohomological lower bound \cite{CFW}, as well as with that given by Farber and Weinberger \cite{FW0, FW1} mentioned above, when the homotopy fibration is a sphere bundle coming from a vector bundle. On the other hand, if $n$ is odd, then we can modify Farber and Weinberger's upper bound \cite{FW1} such that
\begin{equation}
  \label{FWUB}
  \TC[X\to S^{m+1}]\le 2.
\end{equation}
Hence we may ask:

\begin{question}
  \label{question}
  Does there exist a homotopy fibration $S^n\to X\to S^{m+1}$ with $n$ odd and $\TC[X\to S^{m+1}]=2$?
\end{question}

The first purpose of this paper is to provide a sufficient condition for the existence of the homotopy fibration in Question \ref{question}. To state the results, we set notation. As in \cite{JW}, the total space $X$ has a cell decomposition
\[
  X\simeq S^n\cup_\beta e^{m+1}\cup e^{m+n+1}.
\]
We set $\chi(X)=\beta$. Let $\nabla\colon S^n\to S^n\vee S^n$ be the suspension comultiplication of a sphere. Then, for $\alpha\in\pi_{m+n}(S^n)$, we can define the crude Hopf invariant $\overline{H}_\nabla(\alpha)\in\pi_{m+n+1}(S^{2n})$. See \cite[Definition 2.11]{BH} for further details. Let $\eta$ denote the generator of the stable stem $\pi_1^S\cong\Z/2$.

\begin{theorem}
\label{main1}
  Let $2n+1\le m\le 4n-1$ with $n\ge 1$. Suppose that there are $\alpha\in\pi_{m+2n+1}(S^{2n+1})$ and $\beta\in\pi_m(S^{2n+1})$ such that $\overline{H}_\nabla(\alpha)\neq0$, $\Sigma(\eta\circ\beta)=0$ and $[\beta,1_{S^{2n+1}}]=0$. Then there exists a homotopy fibration $S^{2n+1}\to X\to S^{m+1}$ with $\chi(X)=\beta$ and
  \[
    \TC[X\to S^{m+1}]=2.
  \]
\end{theorem}

We prove that Theorem \ref{main1} is applicable to infinitely many cases.

\begin{corollary}
  \label{TC=2 family}
  There exist infinitely many pairs $(m, n)$, e.g. $m-2n-2=0$ and $n$ is odd, such that there exist homotopy fibrations $S^{2n+1}\to X\to S^{m+1}$ with
  \[
  \TC[X\to S^{m+1}] = 2.
  \]
\end{corollary}

To prove these statements, we begin by considering $\TC[X\to S^{m+1}]$ as the sectional category of the fiberwise diagonal map $\Delta\colon X\to X\times_{S^{m+1}}X$. We then recall weak sectional category introduced by Garc\'{i}a-Calcines and Vandembroucq \cite{GV}, which gives a sharper lower bound than the aforementioned lower bounds, but is much harder to compute. As we will see later, the weak sectional category of $\Delta$, denoted by $\mathsf{wsecat}(\Delta)$, coincides with the weak category of a certain $4$-cell complex $W$. Therefore, we aim to provide a sufficient condition for the weak category of $W$ to be equal to $2$.

The next purpose of this paper is to determine the parametrized topological complexity for specific homotopy fibrations $S^n\to X\to S^{n+1}$. We first consider the cases where $n$ is even and the homotopy fibration is the unit sphere bundle of a vector bundle and then consider the $n\equiv-1\pmod4$ with $n\ge11$ cases.

\begin{theorem}
  \label{main 0}
  \begin{enumerate}
    \item Let $S^{2n}\to X\to S^{2n+1}$ be the unit sphere bundle of a real vector bundle with $n\ge 1$. Then
    \[
    \TC[X\to S^{2n+1}]=2.
    \]
    \item Let $S^{4n-1}\to X\to S^{4n}$ be a homotopy fibration with $\chi(X)\equiv 2\pmod 4$ and $n\ge3$. Then
    \[
    \TC[X\to S^{4n}]=2.
    \]
  \end{enumerate}
\end{theorem}

We now concentrate our attention on a specific family of sphere bundles; let $S^n\to T\to S^{n+1}$ be the unit tangent bundle. Then $\chi(T)$ coincides with the Euler characteristic of $S^{n+1}$, which is $2$ for $n$ odd and $0$ for $n$ even. Then, for $n$ even and for $n\equiv-1\pmod4$ with $n\ge11$, we obtain $\TC[T\to S^{n+1}]=2$ by Theorem \ref{main 0}. This equation holds even for $n=3,7$, which we prove separately in Section \ref{3 and 7}. Thus we obtain:

\begin{theorem}
  \label{main3}
  If $n$ is even with $n\ge 2$ or $n\equiv -1\pmod 4$, then $\TC[T\to S^{n+1}]=2$.
\end{theorem}

The final purpose of this paper is to show that there are homotopy fibrations $S^n\to X\to S^{m+1}$ such that $\TC[X\to S^{m+1}]$ cannot be determined by the upper and lower bounds mentioned above. We consider the unit tangent bundle $S^n\to T\to S^{n+1}$ with $n\equiv 1 \pmod 4$. We compute $\mathsf{wsecat}(\Delta)$, or equivalently, the weak category of $W$ and obtain:

\begin{proposition}
  \label{main4}
  There are inequalities
  \[
  \mathsf{wsecat}(\Delta) = 1 \le \TC[T\to S^{4n+2}] \le 2.
  \]
\end{proposition}

For $n=1,2$, $\TC[T\to S^{4n+2}]$ is determined to be $1$. However, for $n=3$, $\TC[T\to S^{4n+2}]$ cannot be determined.

This paper is organized as follows. Section \ref{Parametrized topological complexity} recalls the definition of the parametrized topological complexity and shows \eqref{FWUB}. Section \ref{complex W} investigates the cell structure of the complex $W$. Section \ref{Crude} provides a sufficient condition for $\mathsf{wsecat}(\Delta)=2$ in relation to crude Hopf invariant and Theorem \ref{main1} and Corollary \ref{TC=2 family}. Section \ref{Proof 0} computes the mod $2$ cohomology of $X$ and the total spaces of the related sphere bundles, then proves Theorem \ref{main 0} (1). Section \ref{Proofs} computes $\mathsf{wsecat}(\Delta)$ for specific homotopy fibrations, then proves Theorems \ref{main 0} (2) and \ref{main3}. Section \ref{tangent} also computes $\mathsf{wsecat}(\Delta)$ and proves Proposition \ref{main4}. Finally, Section \ref{Problems} lists some problems that could lead to further investigation.

\subsection*{Acknowledgement}
The author is grateful to Daisuke Kishimoto and Toshiyuki Miyauchi for their valuable advice. The author was supported by JST SPRING, Grant Number JPMJSP2110.

%%%%% Section 2 %%%%%

\section{Parametrized topological complexity}\label{Parametrized topological complexity}

Let $X\to B$ be a fibration and $I = [0, 1]$. The fiberwise path space $X_B^I$ is the subspace of the path space $X^I$ consisting of paths contained in the fiber at some point of $B$. Then the map
\[
  \Pi\colon X_B^I\to X\times_BX,\quad\gamma\mapsto(\gamma(0),\gamma(1))
\]
is a fibration, where $X\times_BX$ denotes the fiberwise product. We consider a path in $X$ as a motion in the configuration space $X$, and the map $X\to B$ represents the constraint. Then an element of $X^I_B$ is a motion in the configuration space $X$ under the constraint given by the map $X\to B$; therefore the section of $\Pi$ is a motion planning under the constraint.

\begin{definition}
  \cite[Definition 4.1]{CFW}
  The parametrized topological complexity of a fibration $X\to B$, denoted by $\TC[X\to B]$, is the minimal integer $k$ such that $B$ can be covered by open sets $U_0,U_1,\ldots,U_k$, over each of which $\Pi$ has a section. If such an integer does not exist, then we set $\TC[X\to B]=\infty$.
\end{definition}

Recall that the sectional category of a map $f\colon X\to Y$, denoted by  $\mathsf{secat}(f)$, is the minimal integer $k$ such that $Y$ can be covered by open sets $U_0,U_1,\ldots,U_k$ over each of which $f$ has a right homotopy inverse. If such an integer does not exist, then we set $\mathsf{secat}(f)=\infty$. We say that two maps $f_1\colon X_1\to Y_1$ and $f_2\colon X_2\to Y_2$ are equivalent if there is a homotopy commutative diagram
\[
  \xymatrix{
    X_1\ar[r]^{f_1}\ar[d]_\simeq&Y_1\ar[d]^\simeq\\
    X_2\ar[r]^{f_2}&Y_2
  }
\]
where the horizontal maps are homotopy equivalences. In this case, $\mathsf{secat}(f_1)=\mathsf{secat}(f_2)$. By definition, $\TC[E\to B]=\mathsf{secat}(\Pi)$. Then since $\Pi$ is equivalent to the fiberwise diagonal map $\Delta\colon X\to X\times_BX$, we get
\[
  \TC[X\to B]=\mathsf{secat}(\Delta\colon X\to X\times_BX).
\]

In particular, we may extend parametrized topological complexity by this equality for maps that are not necessarily fibrations.

We recall two typical lower bounds for parametrized topological complexity. Let $\TC(Z)$ denote the topological complexity of a space $Z$. Namely, $\TC(Z)=\TC[Z\to*]$.

\begin{proposition}
  \cite[(4.4)]{CFW}
  \label{TC fiber}
  Let $F\to X\to B$ be a homotopy fibration. Then
  \[
    \TC[X\to B]\ge\TC(F).
  \]
\end{proposition}

\begin{example}
  \label{lower bound fiber}
  Consider a homotopy fibration $S^n\to X\to B$. It is known that $\TC(S^n)=1$ for $n$ odd and $\TC(S^n)=2$ for $n$ even. Then by Proposition \ref{TC fiber}, we get
  \begin{equation}
    \label{lower bound}
    \TC[X\to B]\ge
    \begin{cases}
      1&n\text{ is odd}\\
      2&n\text{ is even.}
    \end{cases}
  \end{equation}
\end{example}

For an augmented algebra $A$, let $\mathsf{nil}(I)$ denote the nilpotency of the augmented ideal $I$ of $A$. Namely, $\mathsf{nil}(I)$ is the least integer $k$ such that $I^{k+1}=0$.

\begin{proposition}
  \cite[Proposition 7.3]{CFW}
  \label{nil}
  Let $X\to B$ be a map. Then there is an inequality
  \[
    \TC[X\to B]\ge\mathsf{nil}(\mathrm{Ker}\{\Delta^*\colon H^*(X\times_BX)\to H^*(X)\}).
  \]
\end{proposition}

\begin{example}
  \label{even cup-length}
  Let $S^n\to X\to B$ be a homotopy fibration where $n\ge 1$ and $B$ is simply-connected. We consider the Serre spectral sequence for this homotopy fibration. Then
  \[
    E_2\cong H^*(B)\otimes H^*(S^n).
  \]
  Let $u$ denote the generator of $H^n(S^n)\cong\Z$. Then  $a=1\otimes(u\times 1-1\times u)$ is a permanent cycle in the Serre spectral sequence for the homotopy fibration $S^n\times S^n\to X\times_BX\to B$. Clearly, $a\in\mathrm{Ker}\{\Delta^*\colon H^*(X\times_BX)\to H^*(X)\}$ is the generator of the augmented ideal. Note that $a^2=0$ for $n$ odd, $a^2\ne 0$ and $a^3=0$ for $n$ even. Then by Proposition \ref{nil}, we get the inequality \eqref{lower bound}.
\end{example}

Let us recall Farber and Weinberger's upper bound. Let $V\to B$ be a real vector bundle of rank $n+1$, $X\to B$ be the unit sphere bundle, and $\widetilde{X}\to B$ be the bundle of orthonormal $2$-frames in $V$. Then there is a fiber bundle $S^{n-1}\to\widetilde{X}\to X$.

\begin{proposition}
  \cite[Theorem 14]{FW1}
  \label{upper bound FW}
  Let $\widetilde{X}$ and $X$ be as above. Then there is an inequality
  \[
    \TC[X\to B]\le\mathsf{secat}(\widetilde{X}\to X)+1.
  \]
\end{proposition}

We extend this upper bound to a homotopy fibration over a suspension with fiber $S^{2n+1}$.

\begin{proposition}
  \label{upper bound odd}
  Let $S^{2n+1}\to X\to\Sigma B$ be a homotopy fibration with $B$ connected. Then
  \[
    \TC[X\to\Sigma B]\le 2.
  \]
\end{proposition}

\begin{proof}
  Let $C_+B$ and $C_-B$ be the upper and lower cones of $\Sigma B$. Then since $\Sigma B$ is simply-connected, we may think that $X$ is the pushout of the cotriad
  \[
    C_+B\times S^{2n+1}\xleftarrow{(i_+\circ p_1,\mu)}B\times S^{2n+1}\xrightarrow{\mathrm{incl}}C_-B\times S^{2n+1}
  \]
  for some map $\mu$ with $\mu\vert_{*\times S^{2n+1}}=1_{S^{2n+1}}$, where $p_1$ denotes the first projection and $i_+$ denotes the inclusion. Consider the commutative diagram
  \begin{equation}
    \label{Delta pushout}
    \xymatrix{
      C_+B\times S^{2n+1}\ar[d]^{1\times\Delta}&&B\times S^{2n+1}\ar[ll]_{(i_+\circ p_1,\mu)}\ar[r]^{\mathrm{incl}}\ar[d]^{1\times\Delta}&C_-B\times S^{2n+1}\ar[d]^{1\times\Delta}\\
      C_+B\times S^{2n+1}\times S^{2n+1}&&B\times S^{2n+1}\times S^{2n+1}\ar[ll]_{(i_+\circ p_1,\mu\times_B\mu)}\ar[r]^-{\mathrm{incl}}&C_-B\times S^{2n+1}\times S^{2n+1}.
    }
  \end{equation}
  Then the pushuot of the top row is $X$ and the pushout of the second row is $X\times_{\Sigma B}X$ such that the map between the pushouts is the diagonal map $\Delta\colon X\to X\times_{\Sigma B}X$. In particular,
  \[
    X\times_{\Sigma B}X=\Delta(X)\sqcup(C_+B\times A)\sqcup(\mathrm{Int}(C_-B)\times A)
  \]
  where $\Delta(X)$ denotes the diagonal set of $X\times_{\Sigma B}X$, $A$ denotes the configuration space $\{(x, y)\in S^{2n+1}\times S^{2n+1}\mid x\neq y\}$, and $\mathrm{Int}(C_-B)$ denotes the interior of $C_-B$. Since $C_+B$ is contractible and the inclusion
  \[
  \{(x,-x)\in S^{2n+1}\times S^{2n+1}\}\to A
  \]
  is homotopy equivalent, the inclusion $C_+B\times A\to X\times_{\Sigma B}X$ is identified with the composite
  \[
    S^{2n+1}\xrightarrow{(1,\tau)}S^{2n+1}\times S^{2n+1}\xrightarrow{\mathrm{incl}}X\times_{\Sigma B}X,
  \]
  where $\tau\colon S^{2n+1}\to S^{2n+1}$ denotes the antipodal map. Clearly, $\Delta\colon X\to X\times_{\Sigma B}X$ has a section on $\Delta(X)$. Since $\tau$ has degree one, it is homotopic to $1_{S^{2n+1}}$. Then $\Delta\colon X\to X\times_{\Sigma B}X$ has a right homotopy inverse on $C_+B\times A$. We can see that $\Delta\colon X\to X\times_{\Sigma B}X$ has a right homotopy inverse on $\mathrm{Int}(C_-B)\times A$ verbatim. Thus by \cite[Theorem 2.7]{G}, $\TC[X\to \Sigma B]=\mathsf{secat}(\Delta\colon X\to X\times_{\Sigma B}X)\le 2$.
\end{proof}

%%%%% Section 3 %%%%%

\section{The complex $W(\alpha)$}\label{complex W}

We begin by recalling the result of Garc\'{i}a-Calcines and Vandembroucq \cite{GV}. Let $f\colon X\to Y$ be a map and let $Z$ be the cofiber of $f$.
The weak sectional category of $f$, denoted by $\mathsf{wsecat}(f)$, is the minimal integer $k$ such that the following composite is trivial:
\[
Y \xrightarrow{\overline{\Delta}}Y^{\wedge(k+1)}\xrightarrow{g^{\wedge(k+1)}}Z^{\wedge(k+1)},
\]
where $\overline{\Delta}$ denotes the reduced diagonal map, $Y^{\wedge(k+1)}$ denotes the smash product of $k+1$ copies of $Y$ and $g\colon Y\to Z$ denotes the cofiber map of $f$. If such an integer does not exist, then we set $\mathsf{wsecat}(f)=\infty$.

Recall that the weak category of a space $Z$, denoted by $\wcat(Z)$, is the least integer $k$ such that $\overline{\Delta}\colon Z\to Z^{\wedge(k+1)}$ is trivial. If such an integer does not exist, then we set $\wcat(Z)=\infty$. By definition, we clearly have $\mathsf{wsecat}(f)\le\wcat(Z)$.

\begin{proposition}
  \cite[Theorem 21]{GV}
  \label{wcat}
  Let $f\colon X\to Y$ be a map and $Z$ be the cofiber of $f$. If $f$ admits a left homotopy inverse, then there are inequalities
  \[
    \mathsf{nil}(\mathrm{Ker}\{f^*\colon H^*(Y)\to H^*(X)\})\le\mathsf{wcat}(Z)=\mathsf{wsecat}(f)\le\mathsf{secat}(f).
  \]
\end{proposition}

\begin{corollary}
  \label{lower bound wcat}
  For a map $p\colon X\to B$, there are inequalities
  \[
    \mathsf{nil}(\mathrm{Ker}\{\Delta^*\colon H^*(X\times_BX)\to H^*(X)\})\le\mathsf{wcat}(W)=\mathsf{wsecat}(\Delta)\le\TC[X\to B]
  \]
  where $W$ is the cofiber of the fiberwise diagonal map $\Delta$.
\end{corollary}

\begin{proof}
  As the diagonal map $\Delta\colon X\to X\times_BX$ has an obvious left inverse, the statement follows from Proposition \ref{wcat}.
\end{proof}

Let $S^n\to X\to S^{m+1}$ be a homotopy fibration with $m\ge n$. As in \cite[Section 3]{JW}, there is a map $\mu\colon S^m\times S^n\to S^n$ such that $\mu\vert_{*\times S^n}=1_{S^n}$ and $X$ is homotopy equivalent to the homotopy pushout of the cotriad
\begin{equation}
  S^n\xleftarrow{\mu}S^m\times S^n\xrightarrow{p_2}S^n
\end{equation}
\noindent
where $p_2$ denotes the second projection. We then obtain the cell decomposition of $X$.

\begin{lemma}
  \label{cell decomposition E}
  There is a homotopy equivalence
  \[
    X\simeq S^n\cup_{\chi(X)} e^{m+1}\cup e^{m+n+1}
  \]
  such that $X/S^n\simeq S^{m+1}\vee S^{m+n+1}$ and the composite
  \[
    X\to X/S^n\simeq S^{m+1}\vee S^{m+n+1}\xrightarrow{q_1} S^{m+1}
  \]
  is identified with the projection $X\to S^{m+1}$, where $q_1$ denotes the quotient map.
\end{lemma}

Here we note that $\chi(X)=\mu\vert_{S^m\times *}$. Then as in \cite[Theorem 4.4.8]{Ar}, for any map $\widehat{\mu}\colon S^m\times S^n\to S^n$ with $\widehat{\mu}\vert_{S^m\vee S^n}=\chi(X)+1_{S^n}$, there is $\alpha\in\pi_{m+n}(S^n)$ such that $\hat{\mu}$ equals the composite
\[
  \mu(\alpha)\colon S^m\times S^n\to(S^m\times S^n)\vee S^{m+n}\xrightarrow{\mu+\alpha}S^n
\]
where the first map pinches the top cell. Let $X(\alpha)$ denote the homotopy pushout of the cotriad $S^n\xleftarrow{\mu(\alpha)}S^m\times S^n\xrightarrow{p_2}S^n$. Then we get a new homotopy fibration $S^n\to X(\alpha)\to S^{m+1}$ with $\chi(X(\alpha))=\chi(X)$. Moreover, any homotopy fibration $S^n\to\widehat{X}\to S^{m+1}$ with $\chi(\widehat{X})=\chi(X)$ is equivalent to $S^n\to X(\alpha)\to S^{m+1}$ for some $\alpha\in\pi_{m+n}(S^n)$.

We define the complex $W(\alpha)$ as the cofiber of the fiberwise diagonal map
\[
  \Delta\colon X(\alpha)\to X(\alpha)\times_{S^{m+1}}X(\alpha).
\]
In the sequel, we concentrate our attention on the homotopy fibrations
\begin{equation}
  \label{fibration odd}
  S^{2n+1}\to X(\alpha)\to S^{m+1}\quad\text{with}\quad1\le n\quad\text{and}\quad2n+1\le m\le4n-1.
\end{equation}
We now consider the cell decomposition of $W(\alpha)$. We denote the universal Whitehead product $A\star B \to \Sigma A \vee \Sigma B$ by $w$ and the $k$-skeleton of a CW complex $A$ by $A_k$.

\begin{proposition}
  \label{W skeleton}
  There is a homotopy equivalence
  \[
    W(\alpha)_{4n+2}\simeq S^{2n+1}\cup_{-[1_{S^{2n+1}},1_{S^{2n+1}}]}e^{4n+2}.
  \]
\end{proposition}

\begin{proof}
  Recall that $m+1 < 4n+2 < m+2n+2$ by \eqref{fibration odd}. By \eqref{Delta pushout} and Lemma \ref{cell decomposition E},
  \[
  (X(\alpha)\times_{S^{m+1}}X(\alpha))_{4n+2}\simeq(S^{2n+1}\times S^{2n+1})\cup e^{m+1}
  \]
  such that the diagram
  \begin{equation}
    \label{Delta extension}
    \xymatrix{
      X(\alpha)_{4n+2}\ar[r]^(.37)\Delta\ar@{=}[d]&(X(\alpha)\times_{S^{m+1}}X(\alpha))_{4n+2}\ar@{=}[d]\\
      S^{2n+1}\cup e^{m+1}\ar[r]^(.41){\widetilde{\Delta}}\ar[d]&(S^{2n+1}\times S^{2n+1})\cup e^{m+1}\ar[d]\\
      S^{m+1}\ar@{=}[r]&S^{m+1}
    }
  \end{equation}
  commutes up to homotopy, where $\widetilde{\Delta}$ is an extension of the diagonal map $\Delta\colon S^{2n+1}\to S^{2n+1}\times S^{2n+1}$.
  Then $W(\alpha)_{4n+2}$ is the cofiber of $\Delta\colon S^{2n+1}\to S^{2n+1}\times S^{2n+1}$, and so there is a homotopy commutative diagram
  \[
    \xymatrix{
      \ast\ar[r]\ar[d]&S^{4n+1}\ar@{=}[r]\ar[d]_w&S^{4n+1}\ar[d]\\
      S^{2n+1}\ar[r]^(.4)\nabla\ar@{=}[d]&S^{2n+1}\vee S^{2n+1}\ar[d]\ar[r]^(.6){1+(-1)}&S^{2n+1}\ar[d]\\
      S^{2n+1}\ar[r]^(.4)\Delta&S^{2n+1}\times S^{2n+1}\ar[r]&W(\alpha)_{4n+2}
    }
  \]
  such that all rows are homotopy cofibrations, where $\nabla$ denotes the suspension comultiplication. Since the left and middle columns are homotopy cofibrations, so is the right column. Thus since $(1_{S^{2n+1}}+(-1_{S^{2n+1}}))\circ w=[1_{S^{2n+1}},-1_{S^{2n+1}}]=-[1_{S^{2n+1}},1_{S^{2n+1}}]$, the statement follows.
\end{proof}

\begin{corollary}
  \label{W skeleton suspension}
  The complex $W(\alpha)_{4n+2}$ is a suspension, and there is a homotopy equivalence
  \[
    \Sigma (W(\alpha)_{4n+2})\simeq S^{2n+2}\vee S^{4n+3}.
  \]
\end{corollary}

\begin{proof}
  Consider the EHP sequence
  \[
    \pi_{4n}(S^{2n})\xrightarrow{E}\pi_{4n+1}(S^{2n+1})\xrightarrow{H}\pi_{4n+1}(S^{4n+1})\xrightarrow{P}\pi_{4n-1}(S^{2n}).
  \]
  Then $\pi_{4n+1}(S^{2n+1})$ is a finite abelian group, and so the map $H$ is trivial. Then the map $E$ is surjective, implying that the Whitehead product $[1_{S^{2n+1}},1_{S^{2n+1}}]$ is a suspension. Thus the first statement follows from Proposition \ref{W skeleton}. The second statement follows from Proposition \ref{W skeleton} and the fact that the suspension of $[1_{S^{2n+1}},1_{S^{2n+1}}]$ is trivial.
\end{proof}

We describe $W(\alpha)$ as the homotopy pushout of a certain cotriad.
Let $A \ltimes B$ be the half-smash product $A \times B / (A \times \ast)$.

\begin{lemma}
  \label{W}
  The complex $W(\alpha)$ is homotopy equivalent to the homotopy pushout of the cotriad
  \[
    W(\alpha)_{4n+2}\xleftarrow{\tilde{\mu}(\alpha)}S^m\ltimes W(\alpha)_{4n+2}\xrightarrow{p_2}W(\alpha)_{4n+2}
  \]
  such that $\tilde{\mu}(\alpha)\vert_{*\times W(\alpha)_{4n+2}}=1_{W(\alpha)_{4n+2}}$, where $p_2$ denotes the second projection.
\end{lemma}

\begin{proof}
  Recall the commutative diagram \eqref{Delta pushout} such that the fiberwise diagonal map $\Delta\colon X\to X\times_{\Sigma B}X$ is the natural map between the (homotopy) pushout of the two rows.
  On the other hand, as in the proof of Proposition \ref{W skeleton}, $W(\alpha)_{4n+2}$ is the cofiber of the diagonal map $\Delta\colon S^{2n+1}\to S^{2n+1}\times S^{2n+1}$. We set $B = S^m$ and $\mu = \mu(\alpha)$. By taking the cofibers of the columns, we get the cotriad
  \[
    C_+S^m\ltimes W(\alpha)_{4n+2}\xleftarrow{\tilde{\mu}(\alpha)} S^m\ltimes W(\alpha)_{4n+2}\xrightarrow{i_-\ltimes 1} C_-S^m\ltimes W(\alpha)_{4n+2}.
  \]
  Thus the cofiber of $\Delta\colon X(\alpha)\to X(\alpha)\times_{S^{m+1}}X(\alpha)$ is the homotopy pushout of the above cotriad. Moreover, $C_+S^m$ is contractible and the composite
  \[
  S^m\ltimes W(\alpha)_{4n+2} \xrightarrow{i_+\ltimes 1} C_+S^m\ltimes W(\alpha)_{4n+2} \xrightarrow{p_2} W(\alpha)_{4n+2}
  \]
  is the second projection $p_2$. By definition, $\tilde{\mu}(\alpha)\vert_{*\times W(\alpha)_{4n+2}}=1_{W(\alpha)_{4n+2}}$.
  Thus the proof is complete.
\end{proof}

\noindent
By Proposition \ref{W skeleton}, $W(\alpha)_{4n+2}=W(0)_{4n+2}$. Let $i\colon S^{2n+1}\to W(\alpha)_{4n+2}=W(0)_{4n+2}$ denote the bottom cell inclusion, and define
\begin{multline*}
  \Lambda(\alpha)\colon S^m\ltimes S^{2n+1}\xrightarrow{\mathrm{pinch}}(S^m\ltimes S^{2n+1})\vee S^{m+2n+1}\\
  \xrightarrow{\tilde{\mu}(0)\circ(1\ltimes i)+i\circ\alpha}W(0)_{4n+2}=W(\alpha)_{4n+2}.
\end{multline*}

\begin{lemma}
  \label{coaction2}
  $\tilde{\mu}(\alpha)\circ(1_{S^m}\ltimes i) = \Lambda(\alpha)$.
\end{lemma}

\begin{proof}
  Let $q_1\colon S^m\times S^{2n+1}\to S^m\ltimes S^{2n+1}$ be the projection. Then
  \[
    \tilde{\mu}(\alpha)\circ(1_{S^m}\ltimes i)\circ q_1=\Lambda(\alpha)\circ q_1.
  \]
  By Corollary \ref{W skeleton suspension}, $W(\alpha)_{4n+2}$ is a suspension, so the map $q_1^*\colon[S^m\ltimes S^{2n+1},W(\alpha)_{4n+2}]\to[S^m\times S^{2n+1},W(\alpha)_{4n+2}]$ is injective and completes the proof.
\end{proof}

We now give a cell structure of $W(\alpha)$, which has $4$ cells. Recall that, for a suspension space $A$, there is a map $s\colon \Sigma^{m} A\to S^m\ltimes  A$ which induces a homotopy cofibration
\begin{equation}
  \label{s}
  \Sigma^{m} A\xrightarrow{s}S^m\ltimes A\xrightarrow{p_2}A
\end{equation}
\noindent
and is natural with respect to suspensions.

\begin{proposition}
  \label{phi}
  There is a map $\bar{\mu}(\alpha)\colon\Sigma^m(W(\alpha)_{4n+2})\to W(\alpha)_{4n+2}$ such that $W(\alpha)$ is homotopy equivalent to its cofiber.
\end{proposition}

\begin{proof}
  By \eqref{s} and Corollary \ref{W}, there is a homotopy commutative diagram
  \begin{equation}
    \label{W splitting}
    \xymatrix{
      \ast\ar[d]&\Sigma^m (W(\alpha)_{4n+2})\ar[l]\ar[r]^{\bar{\mu}(\alpha)}\ar[d]^s&W(\alpha)_{4n+2}\ar@{=}[d]\\
      W(\alpha)_{4n+2}\ar@{=}[d]&S^m\ltimes W(\alpha)_{4n+2}\ar[l]_{p_2}\ar[r]^{\tilde{\mu}(\alpha)}\ar[d]^{p_2}&W(\alpha)_{4n+2}\ar[d]\\
      W(\alpha)_{4n+2}\ar@{=}[r]&W(\alpha)_{4n+2}\ar[r]&\ast
    }
  \end{equation}
  where all columns are homotopy cofibrations. Now we take the homotopy pushouts of the rows to get a homotopy cofibration
  \[
    Z\to W(\alpha)\to*
  \]
  by Lemma \ref{W}, where $Z$ denotes the cofiber of the map $\bar{\mu}(\alpha)\colon\Sigma^m(W(\alpha)_{4n+2})\to W(\alpha)_{4n+2}$. Thus we obtain $Z\simeq W(\alpha)$, as stated.
\end{proof}

\begin{corollary}
  \label{cell decomposition}
  There are maps $\lambda(\alpha)\colon S^{m+2n+1}\to W(\alpha)_{4n+2}$ and $\phi(\alpha)\colon S^{m+4n+2}\to W(\alpha)_{4n+2}$ such that
  \[
    W(\alpha)\simeq W(\alpha)_{4n+2}\cup_{\lambda(\alpha)}e^{m+2n+2}\cup_{\phi(\alpha)}e^{m+4n+3}.
  \]
\end{corollary}

\begin{proof}
  Combine Corollary \ref{W skeleton suspension} and Proposition \ref{phi}.
\end{proof}

\begin{corollary}
  \label{cat(W)}
  $1\le\mathsf{cat}(W(\alpha))\le2$.
\end{corollary}

\begin{proof}
  By Corollary \ref{W skeleton suspension} and Proposition \ref{phi}, the cone-length of $W(\alpha)$ is at most $2$, implying $\mathsf{cat}(W(\alpha))\le 2$. On the other hand, we have $\mathsf{cat}(W(\alpha))\ge 1$ as $W(\alpha)$ is not contractible.
\end{proof}

We compare the attaching maps $\lambda(\alpha)$ and $\lambda(0)$ of $(m+2n+2)$-cell. By Proposition \ref{W skeleton}, $W(\alpha)_{4n+2}=W(0)_{4n+2}$; thus, both $\lambda(\alpha)$ and $\lambda(0)$ are maps into $W(0)_{4n+2}$.

\begin{proposition}
  \label{coaction3}
  $\lambda(\alpha) = \lambda(0) + i\circ\alpha$.
\end{proposition}

\begin{proof}
  the map $s\colon\Sigma^m (W(\alpha)_{4n+2})\to S^m\ltimes W(\alpha)_{4n+2}$ restricts to a map $\bar{s}\colon S^{m+2n+1}\to S^m\ltimes S^{2n+1}$, so there is a homotopy commutative diagram
  \[
    \xymatrix{
      S^{m+2n+1}\ar[rr]^{\bar{s}}\ar@{=}[d]&&S^m\ltimes S^{2n+1}\ar[d]^{\Lambda(\alpha)}\\
      S^{m+2n+1}\ar[rr]^{\lambda(0) + i\circ\alpha}&&W(\alpha)_{4n+2}.
    }
  \]
  On the other hand, $i$ is a suspension by Corollary \ref{W}. Thus by Lemma \ref{coaction2}, we obtain
  \[
    \lambda(\alpha)=\tilde{\mu}(\alpha)\circ s\circ\Sigma^m i=\tilde{\mu}(\alpha)\circ(1_{S^{2n+1}}\ltimes i)\circ\bar{s} = \Lambda(\alpha)\circ\bar{s}= \lambda(0) + i\circ\alpha,
  \]
  completing the proof.
\end{proof}

%%%%% Section 4 %%%%%

\section{Crude Hopf invariant}\label{Crude}

By Corollary \ref{cat(W)}, we can deduce that $\wcat(W(\alpha))$ is $1$ or $2$. In this section, we provide a sufficient condition for $\wcat(W(\alpha))=2$ in terms of the crude Hopf invariant of the attaching map of a cell of $W(\alpha)$ and prove Theorem \ref{main1} and Corollary \ref{TC=2 family}.

Let us recall the definition of the crude Hopf invariant \cite{BH}. For $k > 1$, let $T^k(Z)$ denote the $k$-fat wedge of $Z$, that is, $T^k(Z)$ is the subspace of $Z^k$ consisting of points $(x_1,\ldots,x_k)$ such that at least one of $x_i$ is the basepoint. Let $F$ denote the homotopy fiber of the inclusion $T^k(Z)\to Z^k$ and let $\widehat{F}$ denote the homotopy fiber of the projection $Z^k\to Z^{\wedge k}$. Then there is a homotopy commutative diagram
\[
  \xymatrix{
    F\ar[r]^i\ar[d]&T^k(Z)\ar[r]\ar[d]&Z^k\ar[r]\ar@{=}[d]&Z^{\wedge k}\ar@{=}[d]\\
    \Omega Z^{\wedge k}\ar[r]&\widehat{F}\ar[r]&Z^k\ar[r]&Z^{\wedge k}.
  }
\]

\begin{lemma}
  \label{retraction}
  There exists a natural left homotopy inverse $r\colon\Omega T^k(Z)\to\Omega F$ of $\Omega i$.
\end{lemma}

Recall that $\cat(Z)\le k-1$ if and only if the diagonal map $Z\to Z^k$ lifts to a map $\gamma\colon Z\to T^k(Z)$. We call such a map $\gamma$ a $k$-th structure map of $Z$.

\begin{definition}
  \cite[Definition 2.11]{BH}
  Let $Z$ be a space with $\cat(Z)=k-1$ and a $k$-th structure map $\gamma$. The \emph{crude Hopf $\gamma$-invariant}
  \[
    \overline{H}_{\gamma}(\zeta)\in\pi_{p+1}(Z^{\wedge k})
  \]
  of $\zeta\in\pi_p(Z)$ is defined by the adjoint of the composite
  \[
    S^{p-1}\xrightarrow{\widehat{\zeta}}\Omega Z\xrightarrow{\Omega\gamma}\Omega T^k(Z)\xrightarrow{r}\Omega F\to\Omega^2Z^{\wedge k}
  \]
  where $\widehat{\zeta}$ denotes the adjoint of $\zeta$.
\end{definition}

\begin{theorem}
  \label{H and wcat}
  \cite[Corollary 3.9]{BH}
  Let $X$ be a space with $\cat(Z)=k-1$ and let $\zeta\in\pi_p(Z)$. If $\wcat(Z\cup_\zeta e^{p+1})=k$, then $\overline{H}_{\gamma}(\zeta)\ne 0$ for any $k$-th structure map $\gamma$.
\end{theorem}

We now apply this theorem to the attaching map $\lambda(\alpha)\colon S^{m+2n+1}\to W(\alpha)_{4n+2}$ of the $(m+2n+2)$-cell of $W(\alpha)$ (see Corollary \ref{cell decomposition}). Let $\rho\colon W(\alpha)_{4n+2}\to S^{4n+2}$ denote the pinch map onto the top cell. Since $\cat (W(\alpha)_{4n+2})=1$ by Corollary \ref{W skeleton suspension}, there is a second structure map $W(\alpha)_{4n+2}\to W(\alpha)_{4n+2}\vee W(\alpha)_{4n+2}$.

\begin{proposition}
  \label{Hopf inv wcat}
  If $\overline{H}_{\gamma}(\lambda(\alpha))\ne 0$ for any second structure map $\gamma$ and $\eta\circ\rho\circ\lambda(\alpha)=0$, then $\wcat(W(\alpha)_{4n+2}\cup_{\lambda(\alpha)}e^{m+2n+2})=2$.
\end{proposition}

\begin{proof}
  By \cite[Theorem 5.14]{BS}, the reduced diagonal map
  \[
  \overline{\Delta}\colon W(\alpha)_{m+2n+2}\to (W(\alpha)_{m+2n+2})^{\wedge2}
  \]
  is homotopic to the composite
  \[
    W(\alpha)_{m+2n+2}\to S^{m+2n+2}\xrightarrow{\overline{H}_{\gamma}(\lambda(\alpha))}(W(\alpha)_{m+2n+2})^{\wedge2}
  \]
  for any second structure map $\gamma$, where the first map is the pinch map onto the top cell. Suppose that $\overline{\Delta}$ is null-homotopic. Then since there is the homotopy cofibration
  \[
    W(\alpha)_{m+2n+2}\to S^{m+2n+2}\xrightarrow{\Sigma\lambda(\alpha)}\Sigma (W(\alpha)_{4n+2}),
  \]
  $\overline{H}_{\gamma}(\lambda(\alpha))$ extends to a map $f\colon\Sigma (W(\alpha)_{4n+2})\to (W(\alpha)_{m+2n+2})^{\wedge2}$. Since the $(4n+3)$-skeleton of $(W(\alpha)_{m+2n+2})^{\wedge2}$ is the bottom cell $S^{4n+2}$, the map $f$ is homotopic to the composite
  \[
    \Sigma (W(\alpha)_{4n+2})\xrightarrow{\Sigma\rho}S^{4n+3}\xrightarrow{c\Sigma\eta}S^{4n+2}\xrightarrow{i}(W(\alpha)_{m+2n+2})^{\wedge2}
  \]
  for some $c\in\Z/2$, where $i$ denotes the bottom cell inclusion. Then we get
  \[
    \overline{H}_{\gamma}(\lambda(\alpha))=f\circ\Sigma\lambda(\alpha)=i\circ\Sigma((c\eta)\circ\rho\circ\lambda(\alpha))=0,
  \]
  which contradicts the assumption. Thus $\overline{\Delta}$ is not null-homotopic, implying
  \[
  \wcat(W(\alpha)_{m+2n+2})\ge 2.
  \]
  On the other hand, since $\cat(W(\alpha)_{4n+2}) = 1$, we have $\wcat(W(\alpha)_{m+2n+2})\le 2$, completing the proof.
\end{proof}

Note that the crude Hopf invariant is determined by the choice of the $k$-th structure map.

\begin{lemma}
  \label{structure map}
  There are exactly two homotopy classes in the second structure maps
  \[
  W(\alpha)_{4n+2} \to W(\alpha)_{4n+2} \vee W(\alpha)_{4n+2}.
  \]
\end{lemma}

\begin{proof}
  Let $F$ denote the homotopy fiber of the inclusion $W(\alpha)_{4n+2}\vee W(\alpha)_{4n+2}\to W(\alpha)_{4n+2}\times W(\alpha)_{4n+2}$. Then since $W(\alpha)_{4n+2}$ is a suspension by Corollary \ref{W skeleton suspension}, there is an exact sequence of groups
  \begin{multline*}
    1\to[W(\alpha)_{4n+2},F]\to[W(\alpha)_{4n+2},W(\alpha)_{4n+2}\vee W(\alpha)_{4n+2}]\\
    \to[W(\alpha)_{4n+2},W(\alpha)_{4n+2}\times W(\alpha)_{4n+2}]
  \end{multline*}
  by Lemma \ref{retraction}. Then the order of the set of the homotopy classes of second structure maps of $W(\alpha)_{4n+2}$ is equal to the order of $[W(\alpha)_{4n+2},F]$. On the other hand, $F$ has the homotopy type of $\Omega(W(\alpha)_{4n+2})\star\Omega(W(\alpha)_{4n+2})$. Then by Proposition \ref{W skeleton} and Corollary \ref{W skeleton suspension}, the $(6n+1)$-skeleton of $F$ is $S^{4n+1}\vee S^{6n+1}\vee S^{6n+1}$, implying
  \[
    [W(\alpha)_{4n+2},F]\cong[W(\alpha)_{4n+2},S^{4n+1}].
  \]
  Moreover, the homotopy cofibration $S^{2n+1}\to W(\alpha)_{4n+2}\xrightarrow{\rho} S^{4n+2}$ induces the exact sequence
  \[
    \pi_{2n+2}(S^{4n+1})\to\pi_{4n+2}(S^{4n+1})\to[W(\alpha)_{4n+2},S^{4n+1}]\to\pi_{2n+1}(S^{4n+1}).
  \]
  Thus since $\pi_{2n+2}(S^{4n+1})=\pi_{2n+1}(S^{4n+1})=0$ and $\pi_{4n+2}(S^{4n+1})\cong\Z/2$, the statement is proved.
\end{proof}

Let $\nabla\colon W(\alpha)_{4n+2}\to W(\alpha)_{4n+2}\vee W(\alpha)_{4n+2}$ denote the suspension comultiplication, and let $\bar{\eta}$ be the composite
\[
  W(\alpha)_{4n+2}\to F\to W(\alpha)_{4n+2}\vee W(\alpha)_{4n+2}
\]
where the first map is the generator of $[W(\alpha)_{4n+2},F]$ and the second map is the fiber inclusion. Then by Lemma \ref{structure map} and its proof, the set of all homotopy classes of second structure map of $W(\alpha)_{4n+2}$ is
\[
  \{\nabla,\,\nabla+\bar{\eta}\}.
\]
From the proof of Lemma \ref{structure map}, we can see that $\bar{\eta}$ is the composite
\begin{equation}
  \label{eta}
  W(\alpha)_{4n+2}\xrightarrow{\rho}S^{4n+2}\xrightarrow{\eta}S^{4n+1}\xrightarrow{w}S^{2n+1}\vee S^{2n+1}\to W(\alpha)_{4n+2}\vee W(\alpha)_{4n+2}.
\end{equation}
For $k=1,2$, let $i_k\colon W(\alpha)_{4n+2}\to W(\alpha)_{4n+2}\vee W(\alpha)_{4n+2}$ denote the $k$-th inclusion.

\begin{lemma}
  \label{left-dist}
  Let $X$ be a finite CW complex of dimension $\le 6n-1$. Then, for $\zeta\in [\Sigma X, W(\alpha)_{4n+2}]$, we have
  \[
  i_1 \circ \zeta + i_2 \circ \zeta
  = (i_1 + i_2) \circ \zeta + [i_1, i_2] \circ H(\zeta)
  \]
  where $H(\zeta)$ is the second James-Hopf invariant of $\zeta$.
\end{lemma}

\begin{proof}
  As in \cite[II, Theorem 2.8]{B}, we have an expansion of $i_1 \circ \zeta + i_2 \circ \zeta$, and by degree reasons, terms higher than two vanish, although $W(\alpha)_{4n+2}$ is not a double suspension, in general. We then obtain equality in the statement.
\end{proof}

\begin{proposition}
  \label{H unique}
  If $\eta\circ \rho\circ \lambda(\alpha) = 0$, then $\overline{H}_{\nabla}(\lambda(\alpha)) = \overline{H}_{\nabla+\bar{\eta}}(\lambda(\alpha))$.
\end{proposition}

\begin{proof}
  By the definition of crude Hopf invariants, it is sufficient to show that $\nabla\circ \lambda(\alpha) = (\nabla+\bar{\eta}) \circ \lambda(\alpha)$. By Lemma \ref{left-dist}, we have
  \[
    \nabla\circ \lambda(\alpha) + \bar{\eta}\circ \lambda(\alpha)
  = (\nabla +\bar{\eta}) \circ \lambda(\alpha) + [\nabla, \bar{\eta}] \circ H(\lambda(\alpha)).
  \]
  By \eqref{eta}, $\bar{\eta}\circ \lambda(\alpha)$ is the composite
  \begin{multline*}
    S^{m+2n+1}\xrightarrow{\lambda(\alpha)}W(\alpha)_{4n+2}\xrightarrow{\rho}S^{4n+2}\xrightarrow{\eta}\\
    S^{4n+1}\xrightarrow{w}S^{2n+1}\vee S^{2n+1}\to W(\alpha)_{4n+2}\vee W(\alpha)_{4n+2},
  \end{multline*}
  which is trivial by assumption. On the other hand, by \eqref{eta}, the Whitehead product $[\nabla, \bar{\eta}]$ is the composite
  \begin{multline*}
    (\Sigma^{-1}W(\alpha)_{4n+2})\star(\Sigma^{-1}W(\alpha)_{4n+2})\xrightarrow{1\star\Sigma^{-1}\rho}\\
    (\Sigma^{-1}W(\alpha)_{4n+2})\star S^{4n+1}\xrightarrow{\omega}W(\alpha)_{4n+2}\vee W(\alpha)_{4n+2}
  \end{multline*}
  for some $\omega$. By degree reasons, the composite
  \[
    S^{m+2n+1}\xrightarrow{H(\lambda(\alpha))}(\Sigma^{-1}W(\alpha)_{4n+2})\star(\Sigma^{-1}W(\alpha)_{4n+2})\xrightarrow{1\star\Sigma^{-1}\rho}(\Sigma^{-1}W(\alpha)_{4n+2})\star S^{4n+1}
  \]
  is trivial, implying $[\nabla, \bar{\eta}] \circ H(\lambda(\alpha))=0$. Thus the proof is complete.
\end{proof}

\begin{corollary}
  \label{wcat=2}
  If $\eta\circ \rho\circ \lambda(\alpha) = 0$ and $\overline{H}_{\nabla}(\lambda(\alpha))\neq 0$, then $\wcat(W(\alpha)) = 2$.
\end{corollary}

\begin{proof}
  By Propositions \ref{Hopf inv wcat} and \ref{H unique}, we get $\wcat(W(\alpha)_{m+2n+2}) = 2$. Recall that there is a commutative diagram
  \[
  \xymatrix{
  W(\alpha)_{m+2n+2}\ar[r]^{\overline{\Delta}_1} \ar[d]^{i} & (W(\alpha)_{m+2n+2})^{\wedge2}\ar[d]^{i^{\wedge2}} \\
  W(\alpha) \ar[r]^{\overline{\Delta}_2}& W(\alpha)^{\wedge2}
  }
  \]
  where $i\colon W(\alpha)_{m+2n+2}\to W(\alpha)$ denotes the inclusion and $\overline{\Delta}_1, \overline{\Delta}_2$ denote the reduced diagonals. Suppose that $\wcat(W(\alpha)) = 1$. Then $\overline{\Delta}_2$ is trivial, and so $\overline{\Delta}_1$ can be lifted to $W(\alpha)_{m+2n+2}\to \widetilde{F}$, where $\widetilde{F}$ denotes the fiber of $i^{\wedge2}$. By the Blakers-Massey theorem, $\widetilde{F}$ is $(m+6n+2)$-connected. Then the lift $W(\alpha)_{m+2n+2}\to \widetilde{F}$ is trivial, which is a contradiction. Thus $\wcat(W(\alpha)) \ge 2$, and by Corollary \ref{cat(W)}, the statement is proved.
\end{proof}

We next aim to determine $\rho\circ \lambda(\alpha)$ for each homotopy fibration \eqref{fibration odd}. Let $\mathsf{Cyl}(f)$ denote the mapping cylinder of a map $f\colon X \to Y$ and let $\iota\colon A\star B\to\Sigma(A\times B)$ be the natural map
\[
  A\star B=(\mathsf{Cyl}(p_1))\cup_{A\times B}(\mathsf{Cyl}(p_2))\to(C_+(A\times B))\cup_{A\times B}(C_-(A\times B))=\Sigma(A\times B).
\]
Then, for a map $\mu\colon A\times B\to C$, its Hopf construction $\mathcal{H}(\mu)\colon A\star B\to \Sigma C$ is the composite
\[
  A\star B\xrightarrow{\iota}\Sigma(A\times B)\xrightarrow{\Sigma\mu}\Sigma C.
\]
In particular, by setting $\mu$ as the projection $q\colon A\times B\to A\wedge B$, we obtain:

\begin{lemma}
  \label{iota retraction}
  The composite
  \[
    A\star B\xrightarrow{\iota}\Sigma(A\times B)\xrightarrow{\Sigma q}\Sigma(A\wedge B)
  \]
  is a homotopy equivalence.
\end{lemma}

Suppose that there is a cotriad
\[
  \xymatrix{
    B& S^m \times B \ar[l]_{p_2} \ar[r]^-{f} & C.
  }
\]
We denote an extension of $f$ by $\tilde{f}\colon S^m\ltimes B\to C\cup e^{m+1}$. Let $Z$ be the homotopy pushout of the above cotriad. Let $q_1\colon A\times B\to A\ltimes B$ denote the projection.

\begin{lemma}
  \label{SigmaC}
  The space $\Sigma Z$ has the homotopy type of the cofiber of the Hopf construction
  \[
    \mathcal{H}(\tilde{f}\circ q_1)\colon S^m \star B \to \Sigma(C \cup e^{m+1}).
  \]
\end{lemma}

\begin{proof}
  We consider a homotopy commutative diagram
  \[
  \xymatrix{
    S^m \star B \ar[rr]^-{(\Sigma q_1)\circ\iota} \ar[d] && \Sigma(S^m \ltimes B) \ar[rr]^{\Sigma \tilde{f}} \ar[d]^{\Sigma p_2}&& \Sigma(C\cup e^{m+1}) \ar[d]\\
    \ast \ar[rr]&& \Sigma B \ar[rr]&& \Sigma Z
  }
  \]
  where the left and right squares are homotopy pushouts. Thus the total square is a homotopy pushout, completing the proof.
\end{proof}

Remark that, if $B$ is a suspension, then the map $\mathcal{H}(q_1)=\Sigma q_1\circ\iota$ is homotopic to $\Sigma s$, where $s$ is the map as in \eqref{s}.

We give a description of the map $\Sigma(\rho\circ\lambda(\alpha))$ in terms of Hopf construction. Let $f$ denote the map $(\mu(\alpha)\times_{S^m}\mu(\alpha))\vert_{S^m\times(S^{2n+1}\vee S^{2n+1})}$.

\begin{lemma}
  \label{Sigma-rho-lambda 1}
   The map $\Sigma(\rho\circ\lambda(\alpha))$ is homotopic to the composite
  \begin{multline*}
    S^m\star S^{2n+1}\xrightarrow{1\star i_1}S^m\star(S^{2n+1}\vee S^{2n+1})\xrightarrow{\mathcal{H}(f)}\\
    \Sigma(S^{2n+1}\times S^{2n+1})\xrightarrow{\Sigma q}\Sigma S^{2n+1}\wedge S^{2n+1}=S^{4n+3}.
  \end{multline*}
\end{lemma}

\begin{proof}
  As in the proof of Proposition \ref{upper bound odd}, there is a homotopy pushout
  \[
    \xymatrix{
      S^m\times(S^{2n+1}\times S^{2n+1})\ar[rr]^-{\mu(\alpha)\times_{S^m}\mu(\alpha)}\ar[d]_{p_2}&&S^{2n+1}\times S^{2n+1}\ar[d]\\
      S^{2n+1}\times S^{2n+1}\ar[rr]&&(X(\alpha)\times_{S^{m+1}}X(\alpha)).
    }
  \]
  Then by Lemma \ref{SigmaC}, the skeleton $\Sigma((X(\alpha)\times_{S^{m+1}}X(\alpha))_{m+2n+2})$ has the homotopy type of the cofiber of the map
  \[
  \mathcal{H}(\tilde{f}\circ q_1)\colon S^m\star(S^{2n+1}\vee S^{2n+1})\to\Sigma((S^{2n+1}\times S^{2n+1})\cup e^{m+1}).
  \]
  Clearly, we can get the similar description of $\Sigma (W(\alpha)_{m+2n+2})$ as the cofiber of $\Sigma\lambda(\alpha)$ such that there is a homotopy commutative diagram
  \[
  \xymatrix{
  S^m\star(S^{2n+1}\vee S^{2n+1}) \ar[rr]^{1\star(1+(-1))}\ar[d]^{\mathcal{H}(\tilde{f}\circ q_1)}&& S^m\star S^{2n+1}\ar[d]^{\Sigma\lambda(\alpha)}\\
  \Sigma((S^{2n+1}\times S^{2n+1})\cup e^{m+1}) \ar[rr]^{\Sigma j}\ar[d]&& \Sigma (W(\alpha)_{4n+2})\ar[d]\\
  \Sigma((X(\alpha)\times_{S^{m+1}}X(\alpha))_{m+2n+2}) \ar[rr]&& \Sigma (W(\alpha)_{m+2n+2}),
  }
  \]
  where $j$ is the cofiber map of $\widetilde{\Delta}$ in \eqref{Delta extension} and the two columns are homotopy cofibrations. On the other hand, there is a homotopy commutative diagram
  \[
    \xymatrix{
      (S^{2n+1}\times S^{2n+1})\ar[d]^i\ar[rr]^{q}&&S^{4n+2}\ar@{=}[d]\\
      (S^{2n+1}\times S^{2n+1})\cup e^{m+1}\ar[d]^j\ar[rr]^{\bar{\rho}}&&S^{4n+2}\ar@{=}[d]\\
      W(\alpha)_{4n+2}\ar[rr]^\rho&&S^{4n+2},
    }
  \]
  where $i$ denotes the inclusion and $\bar{\rho}$ denotes the pinch map onto the top cell. Obviously we have $\mathcal{H}(\tilde{f}\circ q_1) = (\Sigma i)\circ \mathcal{H}(f)$, and then we obtain
  \begin{align*}
    (\Sigma\rho)\circ(\Sigma\lambda(\alpha))&=(\Sigma\rho)\circ(\Sigma\lambda(\alpha))\circ(1\star(1+(-1)))\circ(1\star i_1)\\
    &=(\Sigma\rho)\circ(\Sigma j)\circ\mathcal{H}(\tilde{f}\circ q_1)\circ(1\star i_1)\\
    &=(\Sigma\bar{\rho})\circ\mathcal{H}(\tilde{f}\circ q_1)\circ(1\star i_1)\\
    &=(\Sigma\bar{\rho})\circ(\Sigma i)\circ \mathcal{H}(f)\circ(1\star i_1)\\
    &=(\Sigma q)\circ\mathcal{H}(f)\circ(1\star i_1).
  \end{align*}
  Thus the statement is proved.
\end{proof}

\begin{lemma}
  \label{Theta 3}
  The map $\Sigma(\rho\circ\lambda(\alpha))$ is homotopic to the composite
  \[
  S^m\star S^{2n+1} \xrightarrow{\mathcal{H}(\mu(\alpha),\chi(X)\circ p_1)} \Sigma(S^{2n+1}\times S^{2n+1})\xrightarrow{\Sigma q}\Sigma S^{2n+1}\wedge S^{2n+1}=S^{4n+3}.
  \]
\end{lemma}

\begin{proof}
  By the homotopy commutative diagram
  \[
    \xymatrix{
      S^m\times S^{2n+1}\ar[rr]^{(\mu(\alpha),\chi(X)\circ p_1)}\ar[d]_{1\times i_1}& &S^{2n+1}\times S^{2n+1}\ar@{=}[d]\\
      S^m\times(S^{2n+1}\vee S^{2n+1})\ar[rr]^-{\mu(\alpha)\times_{S^m}\mu(\alpha)}& &S^{2n+1}\times S^{2n+1},
    }
  \]
  we get $\mathcal{H}(f)\circ(1\star i_1)=\mathcal{H}(\mu(\alpha),\chi(X)\circ p_1)$. Thus by Lemma \ref{Sigma-rho-lambda 1}, the proof is complete.
\end{proof}

We then give another description of the map $\rho\circ\lambda(\alpha)$ up to sign.

\begin{theorem}
  \label{Sigma chi}
  There is an equality
  \[
    \Sigma^{2n+1}\chi(X(\alpha))=\pm\rho\circ\lambda(\alpha).
  \]
\end{theorem}

\begin{proof}
  Recall that $\chi(X(\alpha))=\mu(\alpha)\vert_{S^m\times *}$. Then there is a homotopy commutative diagram
  \begin{equation}
    \label{alpha}
    \xymatrix{
      S^m \vee S^{2n+1} \ar[r] \ar[d]_{\nabla \circ \chi +  i_1} & S^m \times S^{2n+1} \ar[r] \ar[d]_{(\mu(\alpha),\chi \circ p_1)} & S^m \wedge S^{2n+1} \ar[d]^{\theta}\\
      S^{2n+1} \vee S^{2n+1} \ar[r] & S^{2n+1} \times S^{2n+1} \ar[r]& S^{2n+1} \wedge S^{2n+1}
    }
  \end{equation}
  \noindent
  where the two rows are homotopy cofibrations. Since $\Sigma \theta$ coincides with the composite in Lemma \ref{Theta 3}, up to sign,
  \[
  \Sigma \theta = \pm \Sigma(\rho \circ \lambda(\alpha)).
  \]
  As $m<4n+2$, $\chi(X)\colon S^m\to S^{2n+1}$ is a suspension by the Freudenthal suspension theorem. Then we have
  \begin{align*}
    [\nabla \circ \chi(X(\alpha)), i_1]
    &= [i_1 \circ \chi(X(\alpha)) + i_2 \circ \chi(X(\alpha)), i_1]\\
    &= [i_1 \circ \chi(X(\alpha)), i_1] + [i_2 \circ \chi(X(\alpha)), i_1] \\
    &= i_1 \circ [\chi(X(\alpha)), 1_{S^{2n+1}}] + [i_2 \circ \chi(X(\alpha)), i_1].
  \end{align*}
  \noindent
  Then since $\chi(X(\alpha)) + 1_{S^{2n+1}} \colon S^m \vee S^{2n+1} \to S^{2n+1}$ extends to $\mu(\alpha)\colon S^m \times S^{2n+1}\to S^{2n+1}$, we have $[\chi(X(\alpha)), 1_{S^{2n+1}}] = 0$, so
  \[
  [\nabla \circ \chi(X(\alpha)), i_1] =  [i_2 \circ \chi(X(\alpha)), i_1].
  \]
  On the other hand, we have $[i_2 \circ \chi(X(\alpha)), i_1]=-[i_1 \circ \chi(X(\alpha)), i_2]$. Then the diagram \eqref{alpha} extends to
  \[
    \xymatrix{
      S^{m-1} \star S^{2n} \ar[r]^{w} \ar[d]_{-\Sigma^{-1}\chi \star 1} & S^m \vee S^{2n+1} \ar[r] \ar[d]^{\nabla \circ \chi +  i_1}& S^m \times S^{2n+1} \ar[r] \ar[d]^{(\mu(\alpha),\chi \circ p_1)} & S^m \wedge S^{2n+1} \ar[d]^{\theta}\\
      S^{2n} \star S^{2n} \ar[r]^{w} & S^{2n+1} \vee S^{2n+1} \ar[r] & S^{2n+1} \times S^{2n+1} \ar[r]& S^{2n+1} \wedge S^{2n+1}.
    }
  \]
  Hence we get
  \[
  \pm \Sigma(\rho \circ \lambda(\alpha)) = \Sigma \theta = \pm \Sigma^{2n+2} \chi(X(\alpha)).
  \]
  Thus since $\rho \circ \lambda(\alpha)\in\pi_{m+2n+1}(S^{4n+2})$ and $m+2n+1\le2(4n+1)$, the statement is proved by the Freudenthal suspension theorem.
\end{proof}

\begin{corollary}
  \label{sufficient condition}
  If $\Sigma(\eta\circ\chi(X(\alpha)))=0$ and $\overline{H}_{\nabla}(\lambda(\alpha))\neq 0$, then $\wcat(W(\alpha)) = 2$.
\end{corollary}

\begin{proof}
Since $m\le 4n-1$, $E^{2n}\colon \pi_{m+1}(S^{2n+1})\to \pi_{m+2n+1}(S^{4n+1})$ is isomorphic by the Freudenthal suspension theorem. Then $\eta\circ\rho\circ\lambda(\alpha) =\pm\eta\circ\Sigma^{2n+1}(\chi(X(\alpha)))$ is trivial if and only if  $\Sigma(\eta\circ\chi(X(\alpha)))$ is trivial. Thus by combining Corollary \ref{wcat=2} and Theorem \ref{Sigma chi}, the statement is proved.
\end{proof}

We are now ready to prove Theorem \ref{main1}. Let $i\colon S^{2n+1}\to W(\alpha)_{4n+2}$ denote the bottom cell inclusion.

\begin{proof}
  [Proof of Theorem \ref{main1}]
  By Lemma \ref{cell decomposition E}, there is a map $\mu\colon S^m\times S^{2n+1}\to S^{2n+1}$ such that $\mu\vert_{S^m\times *} = \beta$ and $X$ is homotopy equivalent to the homotopy pushout of the cotriad
  \[
    S^{2n+1}\xleftarrow{\mu}S^m\times S^{2n+1}\xrightarrow{p_2}S^{2n+1}.
  \]
  By the assumption, $\Sigma(\eta \circ \beta) = 0$. If $\overline{H}_{\nabla}(\lambda(0))$ is nonzero, then $\wcat(W(0)) = 2$ by Corollary \ref{sufficient condition}. Hence by Proposition \ref{upper bound odd} and Corollary \ref{lower bound wcat}, we obtain $\TC[X\to S^{m+1}] = 2$.

  Suppose that $\overline{H}_{\nabla}(\lambda(0))$ is zero. By assumption, $\overline{H}_{\nabla}(\alpha) \neq 0 \in \pi_{m+2n+2}(S^{4n+2})$ and, by Corollary \ref{W skeleton suspension},
  \begin{align*}
    (W(\alpha)_{4n+2})^{\wedge2} &\simeq (W(\alpha)_{4n+2} \wedge S^{2n+1}) \cup_{1_{W(\alpha)_{4n+2}} \wedge [1, -1]} (W(\alpha)_{4n+2} \wedge S^{4n+2}) \\
    &\simeq (S^{4n+2} \vee S^{6n+3}) \vee (S^{6n+3} \vee S^{8n+4}).
  \end{align*}
  Then by the Blakers-Massey theorem, the $(6n+2)$-skeleton of the homotopy fiber of the map $i^{\wedge2}\colon (S^{2n+1})^{\wedge2} \to (W(\alpha)_{4n+2})^{\wedge2}$ is $S^{6n+2}\vee S^{6n+2}$. Since ${m+2n+2} \le 6n+1$, $i^{\wedge2}_*\colon \pi_{m+2n+2}((S^{2n+1})^{\wedge2}) \to\pi_{m+2n+2}((W(\alpha)_{4n+2})^{\wedge2})$ is an isomorphism. Then since $i$ is a suspension by Corollary \ref{W skeleton suspension}, we obtain
  \[
  \overline{H}_{\nabla}(\lambda(\alpha)) = \overline{H}_{\nabla}(\lambda(0) + i_*\alpha) = \overline{H}_{\nabla}(\lambda(0)) + \overline{H}_{\nabla}(i_*\alpha) = i^{\wedge2}_*\overline{H}_{\nabla}(\alpha) \neq 0
  \]
  by Proposition \ref{coaction3}. Thus $\wcat(W(\alpha)) = 2$ by Corollary \ref{sufficient condition}; hence, $\TC[X(\alpha)\to S^{m+1}] = 2$ by Proposition \ref{upper bound odd} and Corollary \ref{lower bound wcat}.
\end{proof}

Next, we prove Corollary \ref{TC=2 family}. Let us recall the relation between the second James-Hopf invariant and the crude Hopf invariant.

\begin{theorem}
  \label{JH and crude}
  \cite[Theorem 5.12.]{BS}
  Let $H\colon\pi_{p}(\Sigma B)\to\pi_{p}(\Sigma (B^{\wedge 2}))$ denote the second James-Hopf invariant and let $\overline{H}_{\nabla}\colon \pi_p(\Sigma B) \to \pi_{p+1}((\Sigma B)^{\wedge2})$ denote the crude Hopf $\nabla$-invariant. Then $\Sigma H = -\overline{H}_{\nabla}$.
\end{theorem}

\begin{lemma}
  \label{Hopf nontrivial family}
 Let $k=1, 2, 5$ or $8i-1$ with $i\ge1$. Then, for each $k$, there are infinitely many $n$ with $2n-3\ge k$ such that $\overline{H}_{\nabla}(\alpha)\neq0$ for some $\alpha\in\pi_{k+4n+2}(S^{2n+1})$.
\end{lemma}

\begin{proof}
  Note that $\overline{H}_{\nabla}(\alpha)\in\pi_{k+4n+3}(S^{4n+2})$ and $k+4n+2\le6n-1 <8n$. Then by Theorem \ref{JH and crude} and the Freudenthal suspension theorem, $E\colon\pi_{k+4n+2}(S^{4n+1})\to\pi_{k+4n+3}(S^{4n+2})$ is isomorphic, and thus, $\overline{H}_{\nabla}(\alpha)\neq0$ if and only if $H(\alpha)\neq0$. Since $k+4n+2\le6n-1$, the EHP sequence
  \[
  \pi_{k+4n+2}(S^{2n+1}) \xrightarrow{H} \pi_{k+4n+2}(S^{4n+1}) \xrightarrow{P} \pi_{k+4n}(S^{2n}) \xrightarrow{E} \pi_{k+4n+1}(S^{2n+1})
  \]
  is exact. We claim that, for each $k$ above, there are infinitely many $n$ such that $[\alpha', 1_{S^{2n}}] = 0$ for some $\alpha' \neq 0 \in \pi_{k+2n+1}(S^{2n})$. Then by the EHP sequence, there exists $\alpha\in\pi_{k+4n+2}(S^{2n+1})$ such that $H(\alpha) = E^{2n+1}\alpha'$. Since $k+2n+1\le4n-2$,
  \[
  E^{2n+1}\colon \pi_{k+2n+1}(S^{2n})\to\pi_{k+4n+2}(S^{4n+1})
  \]
  is an isomorphism by the Freudenthal suspension theorem and we obtain $E^{2n+1}\alpha'\neq0$. Thus the statement follows.

  Let $\nu$ denote the generator of the $2$-component of the stable stem $\pi_3^S\cong\Z/12$. Then by \cite[Chapter V]{To}, $\eta^2, \eta^3$ and $\nu^2$ have order $2$ in the stable stem $\pi_2^S, \pi_3^S$ and $\pi_6^S$, respectively. For $i\ge1$, let $\eta\sigma_i \in \pi_{8i}^S$ denote the element of order $2$ in \cite[Chapter I]{M}. Then we have:
 \begin{itemize}
   \item $[\eta^2, 1_{S^{2n}}] = 0$ if and only if $n \equiv 1 \pmod 2$ by \cite[Corollary 2]{Hi}.
   \item $[\eta^3, 1_{S^{2n}}] = 0$ if and only if $n \equiv 1 \pmod 2$ by \cite[2.10]{Th}.
   \item $[\nu^2, 1_{S^{2n}}] = 0$ if and only if $n \equiv 0 \pmod 4$ by \cite[Theorem 1.3]{KrM} and \cite[2.11]{Th}.
   \item $[\eta\sigma_i, 1_{S^{2n}}] = 0$ if $3\le \nu_2(8i+2n+2)\le 4i$ by \cite[Theorem D]{M}.
 \end{itemize}
 Here $\nu_2(l)$ denotes the $2$-exponent of $l$. Therefore, our claim is proved.
\end{proof}

\begin{proof}
  [Proof of Corollary \ref{TC=2 family}]
  By Lemma \ref{Hopf nontrivial family}, there are infinitely many pairs $(m,n)$ of integers with
  \[
  1\le n,\quad 2n+1\le m\le4n-1\quad\text{and}\quad m-(2n+1)=1, 2, 5\quad\text{or}\quad8i-1,
  \]
  such that $\overline{H}_{\nabla}(\alpha) \neq 0$ for some $\alpha\in\pi_{m+2n+1}(S^{2n+1})$.

  Let $S^{2n+1}\to X\to S^{m+1}$ be the trivial bundle. Then we have $\chi(X) = 0$ and $\TC[X\to S^{m+1}] = 1$. By Corollary \ref{sufficient condition}, it follows that $\wcat(W(\alpha)) = 2$ as $\eta\circ\chi(X(\alpha))=0$. Thus by Proposition \ref{upper bound odd} and Corollary \ref{lower bound wcat}, we obtain $\TC[X(\alpha)\to S^{m+1}] = 2$, completing the proof.
\end{proof}

\begin{remark}
  For the $m-(2n+1)=8i-1$ cases above, there are infinitely many homotopy fibrations $S^{2n+1}\to X\to S^{m+1}$ such that
  \[
  \chi(X)\neq0\quad\text{and}\quad\TC[X\to S^{m+1}] = 2.
  \]
  Namely by \cite{Ad}, there always exists an odd component $\beta\in\pi_{8i-1}^S$ for each $i$. Let $2l+1$ be the order of $\beta$. Then since $[1_{S^{2n+1}},1_{S^{2n+1}}]$ have at most order $2$, we get
  \begin{align*}
    [\beta, 1_{S^{2n+1}}] &= [1_{S^{2n+1}},1_{S^{2n+1}}]\circ \beta\\
    &= [1_{S^{2n+1}},1_{S^{2n+1}}]\circ \beta \circ (2l+1-2l)\\
    &= [1_{S^{2n+1}},1_{S^{2n+1}}]\circ \beta \circ (2l+1)- [1_{S^{2n+1}},1_{S^{2n+1}}]\circ \beta\circ 2l \\
    &= 0 - [1_{S^{2n+1}},1_{S^{2n+1}}]\circ (2l)_{S^{4n+1}}\circ\beta = 0.
  \end{align*}
  Similarly, we get $\eta\circ\beta=0$. Thus by combining Theorem \ref{main1} and Corollaries \ref{sufficient condition} and \ref{Hopf nontrivial family}, we can prove the existence of such homotopy fibrations as stated.
\end{remark}

%%%%% Section 5 %%%%%

\section{Proof of Theorem \ref{main 0} (1)}\label{Proof 0}

Throughout this section, let $V\to S^{2n+1}$ be a vector bundle of rank $2n+1$ with $n\ge 1$. Let $X\to S^{2n+1}$ be the unit sphere bundle of $V$ and let $\widetilde{X}\to S^{2n+1}$ be the orthonormal $2$-frame bundle of $V$. Then there is a fiber bundle $S^{2n-1}\to\widetilde{X}\xrightarrow{p}X$.

For a map $X\to Y$, let $Y\star_XY$ denote the fiberwise join of two copies of $Y$ over $X$.

\begin{proposition}
  \label{secat join}
  \cite[Theorem 3]{Sv}
  For a map $X\to Y$, $\mathsf{secat}(X\to Y)\le 1$ if and only if $Y\star_XY\to X$ has a right homotopy inverse.
\end{proposition}

We aim to show $\mathsf{secat}(p\colon\widetilde{X}\to X)\le 1$. By Proposition \ref{secat join}, this is equivalent to find a right homotopy inverse of the map $p\star_Xp\colon\widetilde{X}\star_X\widetilde{X}\to X$. For the rest of this section, we only consider cohomology with mod $2$ coefficients unless otherwise specified.

\begin{lemma}
  \label{H(E)}
  The mod $2$ cohomology of $X$ is given by
  \[
    H^*(X)=\Lambda(x_{2n},x_{2n+1}),\quad|x_i|=i.
  \]
\end{lemma}

\begin{proof}
  It is sufficient to show that the mod $2$ cohomology Serre spectral sequence for the fiber bundle $S^{2n}\to X\to S^{2n+1}$ collapses at the $E_2$-term, which is equivalent to the Euler class $e(X)$ being trivial. Since $2e(X)=0$ in $H^{2n+1}(S^{2n+1};\Z)\cong\Z$, so $e(X)=0$.
\end{proof}

\begin{lemma}
  \label{H(EE)}
  The mod $2$ cohomology of $\widetilde{X}$ is given by
  \[
    H^*(\widetilde{X})=\Lambda(x_{2n-1},x_{2n},x_{2n+1}),\quad|x_i|=i
  \]
  such that $p^*(x_i)=x_i$ for $i=2n,2n+1$.
\end{lemma}

\begin{proof}
  Let $V_2(\R^{2n+1})$ be the space of orthonormal $2$-frames in $\R^{2n+1}$. Then there is a commutative diagram
  \[
    \xymatrix{
      S^{2n-1}\ar@{=}[r]\ar[d]&S^{2n-1}\ar[r]\ar[d]&\ast\ar[d]\\
      V_2(\R^{2n+1})\ar[r]\ar[d]&\widetilde{X}\ar[r]\ar[d]^p&S^{2n+1}\ar@{=}[d]\\
      S^{2n}\ar[r]&X\ar[r]&S^{2n+1}
    }
  \]
  where all columns and rows are fibrations. We note that the left column is the unit tangent bundle of $S^{2n}$. Then the mod $2$ cohomology Serre spectral sequence for the left column collapses at the $E_2$-term. Thus by Lemma \ref{H(E)} and naturality of spectral sequences, the mod $2$ cohomology Serre spectral sequence for the middle column also collapses at the $E_2$-term. Therefore by Lemma \ref{H(E)}, the statement follows.
\end{proof}

\begin{lemma}
  \label{E*E}
  The mod $2$ cohomology of $\widetilde{X}\star_X\widetilde{X}$ is given by
  \[
    H^*(\widetilde{X}\star_X\widetilde{X})=\Lambda(x_{2n},x_{2n+1},x_{4n-1}),\quad|x_i|=i
  \]
  such that $(p\star_Xp)^*(x_i)=x_i$ for $i=2n,2n+1$.
\end{lemma}

\begin{proof}
  Note that the fiber bundle $S^{2n-1}\to\widetilde{X}\to X$ is the unit sphere bundle of some vector bundle $V'\to X$; thus, the fiber bundle $S^{4n-1}=S^{2n-1}\star S^{2n-1}\to\widetilde{X}\star_X\widetilde{X}\xrightarrow{p\star_Xp}X$ is the unit sphere bundle of the vector bundle $V'\oplus V'\to X$. In particular, $e(\widetilde{X}\star_X\widetilde{X})=e(\widetilde{X})^2$. The proof of Lemma \ref{H(EE)} implies that $e(\widetilde{X})$ is divisible by $2$. Then $e(\widetilde{X}\star_X\widetilde{X})$ is also divisible by $2$, hence the mod $2$ cohomology Serre spectral sequence collapses at the $E_2$-term. Thus the statement follows from Lemma \ref{H(EE)}.
\end{proof}

Consider the homotopy commutative diagram
\[
  \xymatrix{
    \widetilde{X}\ar@{=}[d]&\ast\ar[l]\ar[r]\ar[d]&\widetilde{X}\ar@{=}[d]\\
    \widetilde{X}\ar[d]&\widetilde{X}\times_X\widetilde{X}\ar[l]_{p_1}\ar[r]^{p_2}\ar@{=}[d]&\widetilde{X}\ar[d]\\
    \ast&\widetilde{X}\times_X\widetilde{X}\ar[l]\ar[r]&\ast
  }
\]
such that all columns are cofibrations, where $p_i$ denotes the $i$-th projection. Then by taking the homotopy pushouts of rows, we get a cofibration
\[
  \widetilde{X}\vee\widetilde{X}\xrightarrow{f}\widetilde{X}\star_X\widetilde{X}\xrightarrow{g}\Sigma(\widetilde{X}\times_X\widetilde{X}).
\]

\begin{lemma}
  \label{f}
  The map $g^*\colon H^*(\Sigma(\widetilde{X}\times_X\widetilde{X}))\to H^*(\widetilde{X}\star_X\widetilde{X})$ is surjective for $*=4n-1$ and trivial for $*=4n+1$.
\end{lemma}

\begin{proof}
  Consider the long exact sequence
  \[
    \cdots\to H^*(\Sigma(\widetilde{X}\times_X\widetilde{X}))\xrightarrow{g^*}H^*(\widetilde{X}\star_X\widetilde{X})\xrightarrow{f^*}H^*(\widetilde{X}\vee\widetilde{X})\to\cdots.
  \]
  By Lemma \ref{H(EE)}, $p_1$ and $p_2$ induce an isomorphism
  \[
  H^{4n-1}(\widetilde{X})\oplus H^{4n-1}(\widetilde{X})\cong H^{4n-1}(\widetilde{X}\times_X\widetilde{X}).
  \]
  Then $f^*$ is trivial for $*=4n-1$ and thus the $*=4n-1$ case is proved. By the definition of the map $f$, there is a homotopy commutative diagram
  \[
    \xymatrix{
      \widetilde{X}\vee\widetilde{X}\ar[r]^f\ar[d]^{p\vee p}&\widetilde{X}\star_X\widetilde{X}\ar[d]^{p\star p}\\
      X\vee X\ar[r]&X
    }
  \]
  where the bottom arrow is the folding map. Then by Lemmas \ref{H(E)} and \ref{H(EE)}, the map $f$ is injective in $H^{4n+1}$. Thus by the above long exact sequence, the $*=4n+1$ case is proved as well.
\end{proof}

\begin{proposition}
  \label{upper bound even}
  If $n\ge 2$, then $\mathsf{secat}(p\colon\widetilde{X}\to X)\le 1$.
\end{proposition}

\begin{proof}
  By Proposition \ref{secat join}, it is sufficient to show that $p\star_Xp\colon\widetilde{X}\star_X\widetilde{X}\to X$ has a right homotopy inverse. Clearly, the only obstruction $\mathfrak{o}$ for the existence of such a right homotopy inverse lies in
  \[
    H^{4n+1}(X;\pi_{4n}(S^{4n-1}))\cong H^{4n+1}(X).
  \]
  Suppose $\mathfrak{o}\ne 0$. Then by Lemma \ref{H(E)}, $\mathfrak{o}=x_{2n}x_{2n+1}$, and by Lemma \ref{E*E}, $(p\star_Xp)^*(\mathfrak{o})=x_{2n}x_{2n+1}\ne 0$. On the other hand, since $n\ge 2$, there is $y\in H^{4n-1}(\widetilde{X}\star_X\widetilde{X})$ such that $\mathrm{Sq}^2y=(p\star_Xp)^*(\mathfrak{o})$ by \cite[Chapter III, \S 4]{Sv}, and by Lemma \ref{f}, there is $z\in H^{4n-1}(\Sigma(\widetilde{X}\times_X\widetilde{X}))$ such that $g^*(z)=y$. Then by Lemma \ref{f}, we get
  \[
    0\ne (p\star_Xp)^*(\mathfrak{o})=\mathrm{Sq}^2y=\mathrm{Sq}^2g^*(z)=g^*(\mathrm{Sq}^2z)=0
  \]
  which is a contradiction. Thus $\mathfrak{o}=0$, as desired.
\end{proof}

\begin{proof}
  [Proof of Theorem \ref{main 0} (1)]
  For $n\ge2$, the statement is proved by combining \eqref{lower bound},  Propositions \ref{upper bound FW} and \ref{upper bound even}.
  For $n=1$, all vector bundles over $S^3$ of rank $3$ are trivial as $\pi_3(BO(3)) = 0$. Thus by \cite[Example 4.2]{CFW}, we obtain $\TC[X\to S^3]=\TC(S^2)=2$.
\end{proof}

%%%%% Section 6 %%%%%

\section{Proofs of Theorem \ref{main 0} (2) and Theorem \ref{main3}}\label{Proofs}

\subsection{Proof of Theorem \ref{main 0} (2)}

Throughout this subsection, we consider the homotopy fibrations
\[
S^{2n+1}\to X\to S^{2n+2}\quad\text{with}\quad n\neq1,3.
\]
Namely, we exclude the cases in which the fiber is $S^3$ or $S^7$. For $3\le k<l$, let $f\colon S^l \to S^k$ be a map of order $b<\infty$ and let $Z$ denote the cofiber of $f$. By the Blakers-Massey theorem, the $(k+l-1)$-skeleton of the homotopy fiber of the pinch map $\rho\colon Z \to S^{l+1}$ is homotopy equivalent to $S^k$, so there is an exact sequence
\[
  \pi_{l+1}(S^k) \to\pi_{l+1}(Z) \xrightarrow{\rho_*} \pi_{l+1}(S^{l+1}) \xrightarrow{\delta} \pi_{l}(S^k).
\]
Clearly, the composite
\[
  \pi_l(S^l)\xrightarrow[\cong]{E}\pi_{l+1}(S^{l+1})\xrightarrow{\delta} \pi_{l}(S^k)
\]
represents the induced map $f_*$, where $E$ denotes the suspension map. Then we obtain:

\begin{lemma}
  \label{ker delta}
  $\pi_{l+1}(S^{l+1})/\ker\delta\cong\Z/b$.
\end{lemma}

As in the proof of Corollary \ref{W skeleton suspension}, the Whitehead product $-[1_{S^{2n+1}},1_{S^{2n+1}}]\in\pi_{4n+1}(S^{2n+1})$ is a suspension. Here we consider the order of its desuspension.

\begin{lemma}
  \label{bar w}
  The set $\{\bar{w}\in\pi_{4n}(S^{2n})\mid \Sigma\bar{w}=-[1_{S^{2n+1}},1_{S^{2n+1}}]\}$ has exactly two elements, and each of its element has order
  \[
    \begin{cases}
      2&(n \equiv 0 \pmod 2)\\
      4&(n \equiv 1 \pmod 2).
    \end{cases}
  \]
\end{lemma}

\begin{proof}
  Let $A$ denote the set in the statement. Consider the EHP sequence
  \begin{equation}
    \label{EHP 1}
    \Z/2\cong\pi_{4n+2}(S^{4n+1}) \xrightarrow{P} \pi_{4n}(S^{2n}) \xrightarrow{E} \pi_{4n+1}(S^{2n+1}) \xrightarrow{H} \pi_{4n+1}(S^{4n+1})\cong\Z.
  \end{equation}
  Then by \cite[Corollary 4.18]{HW} , $P(\eta) = [\eta, 1_{S^{2n}}]\ne 0$, and thus the map $P$ is injective. Thus $\mathrm{Ker}\,E = \mathrm{Im}\,P \cong \Z/2$, which implies that $|A|=2$. Take any $a\in A$. Then $A=\{a,a+P(\eta)\}$. If $a$ has order $2^k$ for $k\ge 1$, then as $2P(\eta)=0$, $a+P(\eta)$ also has order $2^k$. Thus it is sufficient to find an element of $A$ having order $2$ for $n \equiv 0 \pmod 2$ and $4$ for $n \equiv 1 \pmod 2$. If $n \equiv 0 \pmod 2$, then by \cite[2.24]{Th}, there is $a\in\pi_{4n}(S^{2n})$ of order $2$ such that
  \[
  Ea=[1_{S^{2n-1}},1_{S^{2n-1}}]=-[1_{S^{2n-1}},1_{S^{2n-1}}].
  \]
  Then we have $a\in A$ and thus the first case is proved.

  Consider the EHP sequence
  \[
    \Z\cong\pi_{4n+3}(S^{4n+3}) \xrightarrow{P} \pi_{4n+1}(S^{2n+1}) \xrightarrow{E} \pi_{4n+2}(S^{2n+2}) \xrightarrow{H} \pi_{4n+2}(S^{4n+3})=0.
  \]
  Then $\mathrm{Ker}\,E=\mathrm{Im}\,P=\Z/2\{[1_{S^{2n+1}},1_{S^{2n+1}}]\}$, where $[1_{S^{2n+1}},1_{S^{2n+1}}]$ has order $2$ as $n\ne 1,3$. By the Freudenthal suspension theorem, $\mathrm{Ker}\,E\cong\mathrm{Ker}\{E^2\colon\pi_{4n+1}(S^{2n+1})\to\pi_{4n+3}(S^{2n+3})\}$. In the EHP sequence \eqref{EHP 1}, the map $H$ is trivial as $\pi_{4n+1}(S^{2n+1})$ is a torsion group; therefore, the map $E$ is surjective. Then the natural map
  \[
  \mathrm{Ker}\{E^3\colon\pi_{4n}(S^{2n})\to\pi_{4n+3}(S^{2n+3})\}\to\mathrm{Ker}\{E^2\colon\pi_{4n+1}(S^{2n+1})\to\pi_{4n+3}(S^{2n+3})\}
  \]
  is surjective too. If $n \equiv 1 \pmod 2$, then  $\mathrm{Ker}\{E^3\colon\pi_{4n}(S^{2n})\to\pi_{4n+3}(S^{2n+3})\}\cong\Z/4$ by \cite[\S 3]{Th}. Thus there is $a\in\pi_{4n}(S^{2n})$ of order $4$ such that $Ea=[1_{S^{2n+1}},1_{S^{2n+1}}]$. Namely, $A$ has an element of order $4$.
\end{proof}

We choose $\bar{w}\in\pi_{4n}(S^{2n})$ such that $\Sigma\bar{w}=-[1_{S^{2n+1}},1_{S^{2n+1}}]$. Let $\overline{W} = S^{2n}\cup_{\bar{w}} e^{4n+1}$. Then by Proposition \ref{W skeleton}, $\Sigma\overline{W}=W(0)_{4n+2}$.

\begin{proposition}
  \label{coker}
  There is an isomorphism
  \[
  \mathrm{Coker} \{E\colon \pi_{4n+1}(\overline{W}) \to \pi_{4n+2}(W(0)_{4n+2})\}
  \cong
  \begin{cases}
    \Z/2 & (n \equiv 1 \pmod 2) \\
    0 & (n \equiv 0 \pmod 2).
  \end{cases}
  \]
\end{proposition}

\begin{proof}
  By the Blakers-Massey Theorem, there is a commutative diagram with exact rows
  \[
    \xymatrix{
      \pi_{4n+1}(S^{2n})\ar[r]^{\bar{i}_*}\ar[d]^E&\pi_{4t+1}(\overline{W})\ar[r]\ar[d]^E&\pi_{4n+1}(S^{4n+1})\ar[r]^{\bar{\delta}}\ar[d]^E&\pi_{4n}(S^{2n})\ar[d]^E\\
      \pi_{4n+2}(S^{2n+1})\ar[r]^{i_*}&\pi_{4n+2}(W(0)_{4n+2})\ar[r]&\pi_{4n+2}(S^{4n+2})\ar[r]^\delta&\pi_{4n+1}(S^{2n+1})
    }
  \]
  where $i$ and $\bar{i}$ are the bottom cell inclusions. Then we get a commutative diagram of exact rows
  \[
    \xymatrix{
      0\ar[r]\ar@{=}[d]&\mathrm{Im}\,\bar{i}_*\ar[r]\ar[d]^f&\pi_{4n+1}(\overline{W})\ar[r]\ar[d]^E&\mathrm{Ker}\,\bar{\delta}\ar[r]\ar[d]^g&0\ar@{=}[d]\\
      0\ar[r]&\mathrm{Im}\,i_*\ar[r]&\pi_{4n+2}(W(0)_{4n+2})\ar[r]&\mathrm{Ker}\,\delta\ar[r]&0.
    }
  \]
  Hence by the snake lemma, we get an exact sequence
  \[
    \mathrm{Coker}\,f\to\mathrm{Coker}\{E\colon \pi_{4n+1}(\overline{W})\to \pi_{4n+2}(W(0)_{4n+2})\}\to\mathrm{Coker}\,g\to 0.
  \]
  Consider the EHP sequence
  \[
  \pi_{4n+1}(S^{2n}) \xrightarrow{E} \pi_{4n+2}(S^{2n+1}) \xrightarrow{H} \pi_{4t+2}(S^{4n+1}) \xrightarrow{P} \pi_{4n}(S^{2n}).
  \]
  As in the proof of Lemma \ref{bar w}, the map $P$ is injective, and thus the map $E$ is surjective. Then $\mathrm{Coker}\,f=0$. By Lemmas \ref{ker delta} and \ref{bar w}, we have
  \[
  \mathrm{Ker}\,\bar{\delta}\cong
  \begin{cases}
    4\Z & (n \equiv 1 \pmod 2) \\
    2\Z & (n \equiv 0 \pmod 2)
  \end{cases}\quad \text{and} \quad \mathrm{Ker}\,\delta\cong 2\Z \subset \Z\cong\pi_{4n+2}(S^{4n+2}).
  \]
  Thus since $E\colon \pi_{4n+1}(S^{4n+1})\to\pi_{4n+2}(S^{4n+2})$ is an isomorphism, $\mathrm{Coker}\,g\cong\Z/2$ for $n\equiv 1\pmod 2$ and $\mathrm{Coker}\,g=0$ for $n\equiv 0\pmod 2$.
\end{proof}

We now consider a homotopy fibration $S^{4n-1}\to X\to S^{4n}$ with $n\ge3$.

\begin{proof}
  [Proof of Theorem \ref{main 0} (2)]
  Let $\chi(X) = 4l+2 \in \pi_{4n-1}(S^{4n-1}) \cong \Z$. Then by Theorem \ref{Sigma chi}, $\rho \circ \lambda(0) = \pm\Sigma^{4n-1}\chi(X) = \pm(4l+2)$. Consider the maps $E\colon\pi_{8n-3}(\overline{W})\to\pi_{8n-2}(W(0)_{8n-2})$ and $g\colon\mathrm{Ker}\,\bar{\delta}\to\mathrm{Ker}\,\delta$ in the proof of Proposition \ref{coker}. By Proposition \ref{coker}, $\rho \circ \lambda(0) = \pm(4l+2) \notin \mathrm{Im}\,g$, implying
  \[
    \lambda(0) \notin \mathrm{Im}\{E\colon\pi_{8n-3}(\overline{W})\to\pi_{8n-2}(W(0)_{8n-2})\}.
  \]
  Then by the EHP sequence
  \[
  \pi_{8n-3}(\overline{W}) \xrightarrow{E} \pi_{8n-2}(W(0)_{8n-2}) \xrightarrow{H} \pi_{8n-2}(\Sigma(\overline{W})^{\wedge2}) \xrightarrow{P} \pi_{8n-4}(\overline{W}),
  \]
  $H(\lambda(0))\ne 0$. Since the $(8n-1)$-skeleton of $\Sigma(\overline{W})^{\wedge2}$ is $S^{8n-3}$, we have
  \[
  \pi_{8n-2}(\Sigma(\overline{W})^{\wedge2})\cong\pi_{8n-2}(S^{8n-3})\cong \Z/2,
  \]
  and by the Freudenthal suspension theorem,
  \[
    \pi_{8n-2}(\Sigma(\overline{W})^{\wedge2}) \xrightarrow{E} \pi_{8n-1}((W(0)_{8n-2})^{\wedge2})
  \]
  is an isomorphism. Then by Theorem \ref{JH and crude}, we get
  \[
  \overline{H}_{\nabla}(\lambda(0)) = -\Sigma H(\lambda(0)) \neq 0.
  \]
  Thus as $\eta\circ\chi(X) = 0$, we obtain $\wcat(W(0)) = 2$ by Corollary \ref{sufficient condition}. Therefore by Proposition \ref{upper bound odd} and Corollary \ref{lower bound wcat}, $\TC[X\to S^{4n}] = 2$.
\end{proof}

As an application, we obtain:

\begin{corollary}
  \label{n equiv 1 lambda}
  For any homotopy fibration $S^{4n+1}\to X\to S^{4n+2}$ with $n\ge1$,
  \[
  \wcat(W(0)_{8n+3}) = 1.
  \]
\end{corollary}

\begin{proof}
  By Proposition \ref{coker}, $E\colon \pi_{8n+1}(\overline{W}) \to \pi_{8n+2}(W(0)_{8n+2})$ is surjective; thus, $\lambda(0) \in \mathrm{Im} E \subset \pi_{8n+2}(W(0)_{8n+2})$. Then $W(0)_{8n+3}$ is a suspension and thus the statement follows.
\end{proof}

\subsection{The $n=3,7$ cases}\label{3 and 7}

Throughout this subsection, we consider the homotopy fibrations
\[
  S^n \to X \to S^{n+1}\quad\text{with}\quad n =  3,7,
\]
and regard $S^3$ and $S^7$ as the unit spheres of $\mathbb{H}$ and $\mathbb{O}$, respectively. Then $S^n$ has a product. We recall from \cite{Ta58} an explicit description of $\pi_n(SO(n+1))$. We define maps $\rho, \sigma\colon S^n \to SO(n+1)$ by
\[
  \rho(x)(y) = xyx^{-1},\quad\sigma(x)(y) = xy
\]
for $x,y\in S^n$. Let $\Z\{a_1,\ldots,a_k\}$ denote the free abelian group generated by $a_1,\ldots,a_k$ and let $\mathbf{r}\colon SU(\frac{n+1}{2}) \to SO(n+1)$ denote the realification map.

\begin{lemma}
  \label{char generators}
  \begin{enumerate}
    \item $\pi_n(SO(n+1))=\Z\{\rho,\sigma\}$.
    \item $\sigma\in\mathrm{Im}\{\mathbf{r}_*\colon\pi_n(SU(\frac{n+1}{2}))\to\pi_n(SO(n+1))\}$.
  \end{enumerate}
\end{lemma}

Let $T\to S^{n+1}$ denote the unit tangent bundle.

\begin{lemma}
  \label{V2 char}
  The map $-\rho + 2\sigma\colon S^n\to SO(n+1)$ determines $T$.
\end{lemma}

\begin{proof}
  Let $V\to S^{n+1}$ be the real vector bundle of rank $n+1$ which is determined by the map $k\rho + l\sigma$.  Then by \cite[Theorems 4.1 and 4.4]{Ta57}, we have
  \[
  p_1(V) = \pm2(2k+l)u \quad\text{for}\quad n=3,\quad\text{and}\quad p_2(V) = \pm6(2k+l)u\quad\text{for}\quad n=7,
  \]
  where $p_i(V)\in H^{4i}(S^{n+1};\Z)$ denotes the $i$-th Pontryagin class and $u$ denotes the generator of $H^{n+1}(S^{n+1};\Z)\cong\Z$.

  Let $X\to S^{n+1}$ be the unit sphere bundle of $V$. Then $X$ has the homotopy type of the homotopy pushout of the cotriad $S^n\xleftarrow{p_2}S^n\times S^n\xrightarrow{\mu}S^n$, where $\mu$ is defined by
  \[
  \mu(x,y) = ((k\rho + l\sigma)(x))(y).
  \]
  Since $\mu(x,1)=x^l$, we get $\chi(X) = \mu\vert_{S^n\times *}=l$, which is the Euler class of $V\to S^{n+1}$. Then, for the tangent bundle $TS^{n+1}\to S^{n+1}$, we get $l = 2$ as the Euler class $e(TS^{n+1})$ is $2$. Furthermore, the total Pontryagin class $p(TS^{n+1})$ is equal to $1$; therefore, we get $k = -1$. Thus $TS^{n+1}\to S^{n+1}$ is determined by $-\rho + 2\sigma$, as stated.
\end{proof}

We define
\[
  \mu\colon S^n\times S^n\to S^n,\quad(x,y)\mapsto xyx=((-\rho + 2\sigma)(x))(y).
\]
Then by Lemma \ref{V2 char}, $T$ has the homotopy type of the homotopy pushout of the cotriad $S^n\xleftarrow{p_2}S^n\times S^n\xrightarrow{\mu}S^n$. We also define
\[
  \hat{\mu}\colon S^n\times S^n\to S^n\quad(x, y)\mapsto x^2y=(2\sigma(x))(y).
\]
Then as $\mu\vert_{S^n\vee S^n}=\hat{\mu}\vert_{S^n\vee S^n}$, there exists $\alpha \in \pi_{2n}(S^n)$ such that $\hat{\mu}=\mu(\alpha)$. Thus by Lemma \ref{char generators}, the homotopy fibration $S^n\to T(\alpha)\to S^{n+1}$ defined by $\hat{\mu}=\mu(\alpha)$ is equivalent to the unit sphere bundle of a complex vector bundle $V\to S^{n+1}$ of rank $\frac{n+1}{2}$.

Recall that $\pi_{2n}(S^n)$ is isomorphic to $\Z/12$ for $n=3$ and $\Z/120$ for $n=7$.

\begin{proposition}
  \label{Samelson}
  The element $\alpha\in\pi_{2n}(S^n)$ is a generator.
\end{proposition}

\begin{proof}
  Let $\omega=\mu(0)\cdot\mu(\alpha)^{-1}\colon S^n\times S^n\to S^n$. Then $\omega$ is given by a map
  \[
    S^n\times S^n\to S^n,\quad(x, y) \mapsto (x,y)\mapsto(xyx)(x^2y)^{-1} = x(yxy^{-1}x^{-1})x^{-1}.
  \]
  Let $\tau\colon S^n\wedge S^n\to S^n\wedge S^n$ denote the map $(x, y)\mapsto (xy, x)$ and $\gamma\colon S^n\wedge S^n\to S^n$ denote the commutator map. Then, for the pinch map $q\colon S^n\times S^n\to S^n\wedge S^n$, it is straightforward to check
  \[
    \omega=\gamma\circ\tau\circ q.
  \]
  On the other hand, we have $\omega=\alpha\circ q$ by definition. As $S^n$ is an H-space, the map $q^*\colon[S^n\wedge S^n,S^n]\to[S^n\times S^n,S^n]$ is injective by \cite[Lemma 1.3.5]{Z}; thus, we get $\alpha=\gamma\circ\tau$. Thus since $\tau$ is a homotopy equivalence, $\alpha$ is a generator of $\pi_{2n}(S^n)$ if and only if $\gamma$ is so. Now  by \cite[\S 9]{J60} and \cite[Theorem II]{Sa}, $\gamma$ generates $\pi_{2n}(S^n)$, so $\alpha$ does too.
\end{proof}

\begin{proof}
  [Proof of Theorem \ref{main3} for $n = 3, 7$]
  By Propositions \ref{coaction3} and \ref{Samelson}, we have $\lambda(\alpha) = \lambda(0) + i_*\alpha$, where $\alpha$ is a generator of $\pi_{2n}(S^n)$ and $i$ is the bottom cell inclusion $S^n\to W(\alpha)_{2n}$. Then by \cite[Chapter V]{To}, we get
  \[
    H(\alpha)  = \eta \neq 0,
  \]
  Therefore we have $\overline{H}_{\nabla}(\alpha)\neq 0$ by Theorem \ref{JH and crude}. Arguing as in the proof of Theorem \ref{main1}, the map $i^{\wedge2}_*\colon \pi_{2n+1}(S^{2n}) \to\pi_{2n+1}((W(\alpha)_{2n})^{\wedge2})$ is an isomorphism, and so we get
  \[
  \overline{H}_{\nabla}(\lambda(\alpha)) - \overline{H}_{\nabla}(\lambda(0)) = \overline{H}_{\nabla}(\lambda(\alpha)-\lambda(0)) = \overline{H}_{\nabla}(i_*\alpha) = i^{\wedge2}_*\overline{H}_{\nabla}(\alpha) \neq 0.
  \]
  Recall that there is a complex vector bundle $V\to S^{n+1}$ such that $T(\alpha)$ is equivalent to its unit sphere bundle. Then $\TC[T(\alpha)\to S^{n+1}] = 1$ by \cite[Corollary 17]{FW1}, implying that $\overline{H}_{\nabla}(\lambda(\alpha))=0$, and so we get  $\overline{H}_{\nabla}(\lambda(0)) \neq 0$. We note that $\Sigma(\eta \circ \chi(T)) = \Sigma(\eta\circ 2) = 0$; then $\wcat(W(0))=2$ by Corollary \ref{sufficient condition}. Thus by Proposition \ref{upper bound odd} and Corollary \ref{lower bound wcat}, $\TC[T\to S^{n+1}] = 2$ as stated.
\end{proof}

\begin{corollary}
  Let $k$ be an integer that is prime to $12$ if $n=3$, prime to $120$ if $n=7$, and let $l$ even. If the bundle $S^n\to X\to S^{n+1}$ is determined by $k\rho+l\sigma$, then $\TC[X\to S^{n+1}] = 2$.
\end{corollary}

\begin{proof}
  Let $S^n\to \widehat{X}\to S^{n+1}$ be the sphere bundle that is determined by $l\sigma$. Then there is an element $\alpha\in\pi_{2n}(S^n)$ such that $\widehat{X} = X(\alpha)$. Arguing as in the proof of Proposition \ref{Samelson}, it follows that $\alpha$ is a generator. Thus since $\chi(X)=l$ is even, the statement is proved in the same way as in the proof of Theorem \ref{main3} for $n = 3, 7$.
\end{proof}

%%%%% Section 7 %%%%%

\section{The unit tangent bundle}\label{tangent}

Throughout this subsection, we consider the unit tangent bundle
\[
  S^{4n+1} \to T \to S^{4n+2}.
\]
Let $\widetilde{T}\to S^{4n+2}$ be the orthonormal $2$-frame bundle of the tangent bundle of $S^{4n+2}$. Then there is a fiber bundle $S^{4n}\to\widetilde{T}\xrightarrow{p} T$, which is equivalent to the fiber bundle
\[
S^{4n}\to V_{3}(\R^{4n+3})\xrightarrow{p} V_{2}(\R^{4n+3}),
\]
where $V_k(\R^{4n+3})$ is the space of the orthonormal $k$-frames in $\R^{4n+3}$ and $p$ is the projection of the first $2$ vectors. Assuming Proposition \ref{main4}, we get an inequality
\[
\wcat(W(0))=1 \le \TC[T\to S^{4n+2}] \le \secat(p)+1
\]
by Proposition \ref{upper bound FW}. On the other hand, by \cite[Theorem 1.15]{J76}, $\secat(p)=0$ if and only if $n=0, 1$. Thus we obtain $\TC[T\to S^{4n+2}] = 1$ for $n=0, 1$. Remark that, for $n\ge2$, $\TC[T\to S^{4n+2}]$ cannot be determined by the upper and lower bounds we have mentioned.

We now aim to prove Proposition \ref{main4}. First, we provide another description of $W(0)$ as the cofiber of a map.

\begin{lemma}
  The complex $W(0)$ has the homotopy type of the cofiber of $p\colon \widetilde{T}\to T$.
\end{lemma}

\begin{proof}
  Let $T_+ \to S^{4n+2}$ and $T_- \to S^{4n+2}$ be the subbundle of $T\times_{S^{4n+2}}T \to S^{4n+2}$ defined by
  \[
  T_+ = \{(x, y)\in T\times_{S^{4n+2}}T\mid \langle x,y\rangle \ge 0\}\quad\text{and}\quad T_- = \{(x, y)\in T\times_{S^{4n+2}}T\mid \langle x,y\rangle \le 0\}
  \]
  where $\langle x,y\rangle$ denotes the fiberwise inner product. Then since $\widetilde{T} = T_+\cap T_-$, $T\times_{S^{4n+2}}T$ is the pushout of the cotriad
  \[
  T_+ \hookleftarrow \widetilde{T} \hookrightarrow T_-.
  \]
  Since $\widetilde{T}$ is obviously a neighborhood deformation retract of $T_-$, the inclusion $\widetilde{T}\to T_-$ is a cofibration. Then $T\times_{S^{4n+2}}T$ is also the homotopy pushout of the cotriad above. Now we have a commutative diagram
  \[
  \xymatrix{
  T_+ \ar[r]^{p_1}\ar[d]& T \ar[d]& T_-\ar[l]_{p_1} \ar[d]\\
  S^{4n+2} \ar@{=}[r]& S^{4n+2} & S^{4n+2}\ar@{=}[l]
  }
  \]
  where the first projection $p_1$ is a homotopy equivalence. Then since the composite
  \[
  T_+\xrightarrow{p_1}T\xrightarrow{\Delta}T\times_{S^{4n+2}}T
  \]
  is homotopic to the inclusion $T_+\to T\times_{S^{4n+2}}T$, we get a homotopy pushout diagram
  \[
  \xymatrix{
  \widetilde{T} \ar[r]\ar[d]^{p}& T \ar[d]^{\Delta}\\
  T \ar[r]& T\times_{S^{4n+2}}T.
  }
  \]
  Thus by taking the cofiber of the vertical maps, the statement is proved.
\end{proof}

We consider the cell decomposition of $W(0)$ in relation to the cell structures of $T$ and $\widetilde{T}$. We recall the cell decomposition of the real Stiefel manifolds by James \cite{J76}.

\begin{lemma}
  \label{V3 cell}
  Let $Q_{k+l,l} = \R P^{k+l}/\R P^{k}$.
  \begin{enumerate}
    \item There is a cell decomposition $T = V_2(\R^{4n+3}) = Q_{4n+2,2}\cup e^{8n+3}$.
    \item There is a cell decomposition
    \[
    \widetilde{T} = V_3(\R^{4n+3}) = (Q_{4n+2,3} \cup e^{8n+1} \cup e^{8n+2}) \cup_\psi e^{8n+3} \cup e^{12n+3}.
    \]
    \item The map $p\colon \widetilde{T}\to T$ is a cellular map.
  \end{enumerate}
\end{lemma}

We also recall that there are cell decompositions
\[
Q_{4n+2,2} = S^{4n+1} \cup_2 e^{4n+2}\quad\text{and}\quad Q_{4n+2,3} = S^{4n}\vee (S^{4n+1} \cup_2 e^{4n+2}).
\]
From the construction, it follows that $\widetilde{T}_{8n+2}$ is the pullback of the triad
\[
Q_{4n+2,2} \hookrightarrow T \xleftarrow{p} \widetilde{T}.
\]
Since $Q_{4n+2,2}$ is a suspension and $Q_{4n+2,3}$ is the cofiber of the trivial map $S^{4n}\cup_2 e^{4n+1}\to S^{4n}$, there is a map $\mu\colon (S^{4n}\cup_2 e^{4n+1}) \times S^{4n}\to S^{4n}$ with
\[
\mu\vert_{(S^{4n}\cup e^{4n+1})\vee S^{4n}} = 0+1_{S^{4n}}
\]
such that $\widetilde{T}_{8n+2}$ has the homotopy type of the homotopy pushout of the cotriad
\[
S^{4n} \xleftarrow{p_2} (S^{4n}\cup_2 e^{4n+1}) \times S^{4n} \xrightarrow{\mu} S^{4n}.
\]
Let $W'$ be the cofiber of $p\colon\widetilde{T}_{8n+2} \to Q_{4n+2,2}$. Clearly, $\mathrm{dim}\,W' = 8n+3$ and $W'$ is $4n$-connected.
\begin{lemma}
  \label{cat W'}
  $\cat(W') = 1$.
\end{lemma}

\begin{proof}
  Consider the homotopy commutative diagram
  \[
  \xymatrix{
  \ast \ar[d]&& (S^{4n}\cup_2 e^{4n+1}) \ar[ll]\ar[rr]\ar[d]^{i_1}&& \ast \ar[d]\\
  S^{4n} \ar[d]&& (S^{4n}\cup_2 e^{4n+1}) \times S^{4n} \ar[ll]_-{p_2}\ar[rr]^-{\mu}\ar[d]^{p_1} && S^{4n} \ar[d]\\
  \ast && (S^{4n}\cup_2 e^{4n+1}) \ar[ll]\ar[rr]&& \ast
  }
  \]
  Note that, for each column, the composite of the two vertical maps is the identity. Then by taking the homotopy pushout of each row, we get ${Q_{4n+2,2}}\to\widetilde{T}_{8n+2}\xrightarrow{p}{Q_{4n+2,2}}$, whose composite is homotopic to the identity map. Then the map $p\colon\widetilde{T}_{8n+2}\to{Q_{4n+2,2}}$ admits a right homotopy inverse. By \cite[Theorem 21]{GV}, we get $\wcat(W') \le 1$. Then by definition of the weak category, the reduced diagonal map  $\overline{\Delta}\colon W' \to W'^{\wedge2}$ is trivial. On the other hand, by the Blakers-Massey theorem, the $(12n+1)$-skeleton of the fiber of the quotient map $W'^2 \to W'^{\wedge2}$ is $W'^2\vee W'^2$. Since $\mathrm{dim}\,W' = 8n+3$, there is a lift $W'\to W'^2\vee W'^2$ of the diagonal map $\Delta\colon W'\to W'^2\times W'^2$. Then we obtain $\cat(W')\le1$. Thus since $W'$ is not contractible, the statement is proved.
\end{proof}

\begin{lemma}
  \label{delta retraction}
  The connecting map $\delta\colon W' \to \Sigma(\widetilde{T}_{8n+2})$ of the homotopy cofibration
  \[
  \widetilde{T}_{8n+2} \xrightarrow{p} Q_{4n+2,2}\to W'
  \]
  admits a left homotopy inverse.
\end{lemma}

\begin{proof}
  As in the proof of Lemma \ref{cat W'}, $p$ admits a right homotopy inverse, and so does $\Sigma^2 p\colon\Sigma^2(\widetilde{T}_{8n+2}) \to \Sigma^2 Q_{4n+2,2}$. By the Blakers-Massey theorem, the $(8n+5)$-skeleton of the fiber of $\Sigma^2 Q_{4n+2,2} \to \Sigma^2 W'$ is $\Sigma^2(\widetilde{T}_{8n+2})$. Then by the existence of a right homotopy inverse, $\Sigma^2 Q_{4n+2,2} \to \Sigma^2 W'$ is null-homotopic. Since $\mathrm{dim}\,Q_{4n+2,2} = 4n+2$ and $W'$ is $4n$-connected, $Q_{4n+2,2} \to W'$ is trivial by the Freudenthal suspension theorem. Thus the statement is proved.
\end{proof}

Furthermore, the pushout construction of a spherical fibration in the proof of Proposition \ref{upper bound odd} can be generalized; there is a map $\kappa\colon S^{8n+2}\times S^{4n} \to \widetilde{T}_{8n+2}$ with $\kappa\vert_{*\times S^{4n}}$ being the inclusion of the bottom cell, such that there is a homotopy commutative diagram
\[
\xymatrix{
S^{4n} \ar[d]&& S^{8n+2}\times S^{4n} \ar[ll]_{p_2} \ar[rr]^{\kappa}\ar[d]^{p_1}&& \widetilde{T}_{8n+2}\ar[d]^{p}\\
\ast \ar[d]&& S^{8n+2} \ar[ll]\ar[rr]\ar[d]&& Q_{4n+2,2}\ar[d]\\
S^{4n+1} \ar@{=}[d]&& S^{8n+2}\ltimes S^{4n+1} \ar[ll]_{p_2} \ar[rr]^{\tilde{\kappa}}\ar[d]&& W'\ar[d]^{\delta}\\
S^{4n+1} && \Sigma(S^{8n+2}\times S^{4n}) \ar[ll]_{\Sigma p_2} \ar[rr]^{\Sigma\kappa}&& \Sigma(\widetilde{T}_{8n+2})
}
\]
where all columns are homotopy cofiber sequences and $p_1\colon\widetilde{T}\to T$ is the map between the homotopy pushout of the top two rows. By taking the homotopy pushouts of the rows, we get a cofiber sequence $\widetilde{T} \xrightarrow{p} T \to W(0) \xrightarrow{\delta} \Sigma\widetilde{T}$. Then by the argument as in the proof of Proposition \ref{phi}, $W(0)$ has the homotopy type of the cofiber of the composite
\[
S^{12n+3}=S^{8n+2} \star S^{4n} \xrightarrow{s} S^{8n+2} \ltimes S^{4n+1} \xrightarrow{\tilde{\kappa}} W',
\]
where $s$ is the map as in \eqref{s}. Note that $\kappa\vert_{S^{8n+2}\times *}$ corresponds to the attaching map $\psi$ of the $(8n+3)$-cell of $\widetilde{T}$.

\begin{lemma}
  \label{phi quotient}
  The composite
  \[
  S^{8n+2} \star S^{4n} \xrightarrow{s} S^{8n+2} \ltimes S^{4n+1} \xrightarrow{\tilde{\kappa}} W' \xrightarrow{q} W'/W'_{4n+1} = S^{8n+2}\cup_2 e^{8n+3}
  \]
  is trivial.
\end{lemma}

\begin{proof}
  Consider the homotopy commutative diagram
  \[
  \xymatrix{
  Q_{4n+2,3} \ar[r] \ar[d]& \widetilde{T}_{8n+2} \ar[r]^-{q} \ar[d]^{p}& S^{8n+1}\cup_2 e^{8n+2} \ar[d]\\
  Q_{4n+2,2} \ar@{=}[r] \ar[d]& Q_{4n+2,2} \ar[r] \ar[d]& \ast\ar[d]\\
  S^{4n+1} \ar[r] \ar[d]& W' \ar[r]^-{q} \ar[d]^{\delta}& S^{8n+2}\cup_2 e^{8n+3}\ar@{=}[d]\\
  \Sigma Q_{4n+2,3} \ar[r]& \Sigma(\widetilde{T}_{8n+2}) \ar[r]^-{q}& S^{8n+2}\cup_2 e^{8n+3}
  }
  \]
  where all columns and rows are homotopy cofiber sequences. Then the composite in the statement corresponds to the attaching map of the $(12n+4)$-cell of $\Sigma\widetilde{T}/\Sigma Q_{4n+2,3}$. On the other hand, by \cite{J76}, there is a stable homotopy eqivalence
  \begin{equation}
    \label{V3 stable splitting}
    \widetilde{T} \simeq_{S} Q_{4n+2,3} \vee (S^{8n+1} \cup_2 e^{8n+2}) \vee S^{8n+3} \vee S^{12n+3}.
  \end{equation}
  \noindent
  Then by the Freudenthal suspension theorem, we get
  \[
  \Sigma\widetilde{T}/\Sigma Q_{4n+2,3} =  (S^{8n+2} \cup_2 e^{8n+3}) \vee S^{8n+4} \vee S^{12n+4}.
  \]
  Thus the attaching map of the $(12n+4)$-cell of $\Sigma\widetilde{T}/\Sigma Q_{4n+2,3}$ is trivial, completing the proof.
\end{proof}

We are now ready to prove Proposition \ref{main4}.

\begin{proof}
  [Proof of Proposition \ref{main4}]
  Recall that there is a homotopy commutative diagram
  \[
  \xymatrix{
  S^{12n+3} \ar[rr]^{s}\ar@{=}[d]&& S^{8n+2}\ltimes S^{4n+1} \ar[rr]^{\tilde{\kappa}}\ar[d]&& W'\ar[d]^{\delta}\\
  S^{12n+3} \ar[rr]^{\iota}&& \Sigma(S^{8n+2}\times S^{4n}) \ar[rr]^{\Sigma\kappa}&& \Sigma(\widetilde{T}_{8n+2}).
  }
  \]
  By \cite[Theorem 3.16]{BS}, we have
  \[
  \overline{H}_{\nabla}(\Sigma\kappa\circ\iota) = \overline{H}_{\nabla}(\delta\circ\tilde{\kappa}\circ s) = \delta^{\wedge2}\circ \overline{H}_{\nabla}(\tilde{\kappa}\circ s)+ \overline{H}_{\nabla}(\delta)\circ \Sigma(\tilde{\kappa}\circ s).
  \]
  We aim to prove $\overline{H}_{\nabla}(\tilde{\kappa}\circ s) = 0$. Then since $\cat(W')=1$ by Lemma \ref{cat W'}, it follows that $\wcat(W(0)) = 1$ by Theorem \ref{H and wcat} as stated.

  First, we prove that $\overline{H}_{\nabla}(\delta)\circ \Sigma(\tilde{\kappa}\circ s) = 0$. Since $(\Sigma(\widetilde{T}_{8n+2}))^{\wedge2}$ is $(8n+1)$-connected, $\overline{H}_{\nabla}(\delta)$ has an extension
  \[
  \tilde{\delta}\colon \Sigma W'/S^{4n+2} = S^{8n+3}\cup_2 e^{8n+4} \to (\Sigma(\widetilde{T}_{8n+2}))^{\wedge2}.
  \]
  On the other hand, the composite
  \[
  S^{12n+4} \xrightarrow{\Sigma s} \Sigma(S^{8n+2} \ltimes S^{4n+1}) \xrightarrow{\Sigma\tilde{\kappa}} \Sigma W' \xrightarrow{q} S^{8n+3}\cup_2 e^{8n+4}
  \]
  is trivial by Lemma \ref{phi quotient}. Then we get
  \[
  \overline{H}_{\nabla}(\delta)\circ \Sigma(\tilde{\kappa}\circ s) = \tilde{\delta}\circ q\circ \Sigma(\tilde{\kappa}\circ s) = 0.
  \]

   We now have $\overline{H}_{\nabla}(\Sigma\kappa\circ\iota) = \delta^{\wedge2}\circ \overline{H}_{\nabla}(\tilde{\kappa}\circ s)$. Since $\kappa\vert_{S^{8n+2}\vee S^{4n}} = \psi + i$, where $i$ is the bottom cell inclusion, we get
   \[
   \overline{H}_{\nabla}(\Sigma\kappa\circ\iota) = \overline{H}_{\nabla}(\mathcal{H}(\kappa)) = \pm(\Sigma \psi \wedge \Sigma i) = \pm(1_{\Sigma(\widetilde{T}_{8n+2})}\wedge \Sigma i)\circ(\Sigma\psi\wedge 1_{S^{4n+1}}).
   \]
   On the other hand, it follows that $\Sigma\psi\wedge 1_{S^{4n+1}} = \Sigma^{4n+2}\psi$ is trivial from the stable splitting \eqref{V3 stable splitting} and the Freudenthal suspension theorem. Then we get $\overline{H}_{\nabla}(\Sigma\kappa\circ\iota) = 0$. Thus by Lemma \ref{delta retraction}, we obtain $\overline{H}_{\nabla}(\tilde{\kappa}\circ s) = 0$, completing the proof.
\end{proof}

%%%%% Section 8 %%%%%

\section{Open problems}\label{Problems}

We conclude this paper with a list of questions regarding the parametrized topological complexity of spherical fibrations over spheres. We hope that these problems precipitate further research.

First, we consider a homotopy fibration $S^{2n+1}\to X\to S^{m+1}$ with $1\le n$ and $2n+1\le m\le4n-1$. Note that, except for the cases in Section \ref{tangent}, we have been considered only $\wcat(W(0)_{m+2n+2})$ as a lower bound.

\begin{problem}
  Does there exist a homotopy fibration $S^{2n+1}\to X\to S^{m+1}$ such that $\wcat(W(0)_{m+2n+2}) = 1$ but $\wcat(W(0)) = 2$?
\end{problem}

In Section \ref{tangent}, we have considered the unit tangent bundle $S^{4n+1}\to T\to S^{4n+2}$ and proved $\wcat(W(0)) = 1$ as an example of the limit of the lower bound by the weak category.

\begin{problem}
  For $n\ge 2$, is $\TC[T\to S^{4n+2}] = 1$ or $2$?
\end{problem}

Next, we consider a homotopy fibration $S^{2n}\to X\to S^{m+1}$ with $2n\le m\le4n-3$. Then by Lemma \ref{W}, the space $W$ in Corollary \ref{lower bound wcat} can be described as a $4$-cell complex $S^{2n}\cup e^{4n}\cup e^{m+2n+1}\cup e^{m+4n+1}$. However, the computation of $\wcat(W)$ is considerably more difficult. We have $\wcat(W_{4n}) = 2$ since $-[1_{S^{2n}}, 1_{S^{2n}}]$ is not a suspension. Then we get $\wcat(W_{m+2n+1}) = 2$ by degree reason.

\begin{problem}
  Does there exist a homotopy fibration $S^{2n} \to X\to S^{m+1}$ such that $\wcat(W)=3$?
\end{problem}

We also need to improve an upper bound; the upper bound in Proposition \ref{upper bound FW} can only be applied to the unit sphere bundles, and the computation is not easy for $m>2n$ cases.

\begin{problem}
  Generalize the upper bound in Proposition \ref{upper bound FW} to homotopy fibrations $S^{2n}\to X\to S^{m+1}$.
\end{problem}

Finally, we consider a homotopy fibration $S^{n}\to X\to S^{m+1}$ with $2n-3<m$. Then the argument about the crude Hopf invariants in Section \ref{Crude} cannot be applied.

\begin{problem}
  Compute $\wcat(W)$ and $\TC[X\to S^{m+1}]$ for homotopy fibrations $S^n \to X\to S^{m+1}$ with $2n-3<m$.
\end{problem}

\end{document}